\documentclass{amsart}
\usepackage[utf8]{inputenc}

\usepackage{mathrsfs} 
\usepackage{amssymb}
\usepackage{bm}
\usepackage{graphicx}
\usepackage[centertags]{amsmath}
\usepackage{amsfonts}
\usepackage{amsthm}
\usepackage{enumerate}
\usepackage{tocvsec2}
\usepackage{xcolor}
\usepackage[margin=.75in]{geometry}

\newtheorem{theorem}{Theorem}[section]
\newtheorem*{theorem*}{Theorem}

\newtheorem{lemma}[theorem]{Lemma}
\newtheorem{rem}[theorem]{Remark}

\newtheorem{proposition}[theorem]{Proposition}

\newtheorem*{fact*}{Fact}

\newtheorem{claim}{Claim}
\theoremstyle{definition}

\newcommand{\ee}{\varepsilon}
\newcommand{\nn}{\mathbb{N}}
\newcommand{\rr}{\mathbb{R}}

\begin{document}

\title{Lipschitz subtype}
\author{R.M. Causey}

\begin{abstract} We give necessary and sufficient conditions for a Lipschitz map, or more generally a uniformly Lipschitz family of maps, to factor the Hamming cubes. This is an extension to Lipschitz maps of a particular spatial result of Bourgain, Milman, and Wolfson \cite{BMW}.

\end{abstract}
\maketitle

\section{Introduction}

In $1969$, M. Ribe \cite{Ribe} proved that two Banach spaces which are uniformly homeomorphic must be crudely finitely representable in each other. Since then, the Ribe program has attracted significant attention (see \cite{Naor} for a survey on the Ribe program),  with the goal of providing purely metric characterizations of local properties of Banach spaces.   An important result in this area is that of Bourgain, Milman, and Wolfson, who defined one notion of metric type $p$ and proved that any family of metric spaces with no non-trivial type must contain almost isometric copies of the Hamming cubes. Another goal within the Ribe program is to find, for a given class of important linear operators between Banach spaces, natural metric analogues within the class of Lipschitz maps between metric spaces (see \cite{ChavezDominguez}\cite{FJ}, \cite{JS}).  One such class is the class of super-Rosenthal operators, for which Beauzamy \cite{Beauzamy} gave a linear characterization in terms of a sequence of subtype constants (we discuss the notion of subtype in Section $2$).    The goal of this work is to undertake the process of proving the Lipschitz analogue of Beauzamy's linear result for the super-Rosenthal operators. We define different notions of subtype constants for a Lipschitz map (or more generally, a uniformly Lipschitz collection of maps) between  Banach spaces, which are the analogues of linear subtype constants appearing in the literature in the aforementioned work of Beauzamy and the work of Hinrichs \cite{Hinrichs}.   We prove the non-linear analogues of the results found in the work of Beauzamy and Hinrichs, in that if the subtype constants of a uniformly Lipschitz family of maps exhibit the asymptotically worst possible behavior, then the Lipschitz maps preserve copies of the Hamming cubes. Our subtype constants are based on the Bourgain, Milman, Wolfson notion of metric type. We next make these descriptions precise, and then state the main result.

We agree to the convention that $\frac{0}{0}=0$. Given a map $g:(U, d_U)\to (V, d_V)$ between metric spaces, we let $$\text{Lip}(g)=\sup_{x,y\in U}\frac{d_V(g(x), g(y))}{d_U(x,y)}.$$  For a map $g:(U, d_U)\to (V, d_V)$, we let $\text{dist}(g)=\infty$ if $g$ is not injective, and otherwise we let $\text{dist}(g)=\text{Lip}(g)\text{Lip}(g^{-1})$, where $g^{-1}$ is understood to be defined on $g(U)$.

We let $2^n=\{\pm 1\}^n$ be the (vertex set of the) Hamming cube.  Given $\ee\in 2^n$, we denote the coordinates of $\ee$ by $\ee(1), \ee(2), \ldots$. We endow $2^n$ with the normalized graph metric $$\partial_n (\ee, \delta)=\frac{1}{n}|\{i: \ee(i)\neq \delta(i)\}|.$$    When no confusion can arise, we will suppress the subscript $n$ and just write $\partial$. We also endow $2^n$ with the uniform probability measure $\mathbb{P}_n$, also suppressing the subscript when no confusion can arise. Given $1\leqslant i\leqslant n$, we let $d_i$ denote the function on $2^n$ which changes the $i^{th}$ coordinate and leaves the other coordinates unchanged.

To avoid cumbersome notation, if $(X, d_X)$, $(Y, d_Y)$ are metric spaces and $f:2^n\to X$, $F:X\to Y$ are functions, we let $\varrho^f_X$, $\varrho^f_Y$, respectively, denote the pseudometrics on $2^n$ given by $$\varrho^f_X(\ee, \delta)= d_X(f(\ee), f(\delta))$$ and $$\varrho^f_Y(\ee, \delta)=d_Y(F\circ f(\ee), F\circ f(\delta)).$$

Now suppose we have $\lambda>0$ fixed and a collection $\mathcal{F}$ of $\lambda$-Lipschitz functions between (possibly different) metric spaces.    For $n\in\nn$, we let $a_n(\mathcal{F})$ denote the infimum of those $a>0$ such that for each $F:X\to Y\in \mathcal{F}$ and $f:2^n\to X$, $$\mathbb{E}\varrho_Y^f(\ee, -\ee) \leqslant a\text{Lip}(f:2^n\to X).$$  Note that $a\leqslant \lambda$.     These are the Lipschitz analogues of the linear quantities appearing in \cite{Beauzamy}.  We note that there appears a factor of $n^{-1}$ on the expectation. The reason is because this constant $n^{-1}$ has been subsumed by $\text{Lip}(g)$ and our convention of using the normalized graph metric on $2^n$.

For $1<p<\infty$, let $b_{p,n}(\mathcal{F})$ denote the infimum of those $b>0$ such that for each $F:X\to Y\in \mathcal{F}$ and $f:2^n\to X$,  $$\mathbb{E}\varrho_Y^f(\ee, -\ee)^p \leqslant b^p n^{p-1}\sum_{i=1}^n \mathbb{E}\varrho_X^f(\ee, d_i\ee)^p.$$  Let us note that by combining the triangle and H\"{o}lder inequalities, $b_{p,n}(\mathcal{F})\leqslant \lambda$.    In the case $p=2$, these are the Lipschitz analogues of the linear quantities appearing in \cite{Hinrichs}, as well as the generalization to maps of the metric type $1$ constant as defined by Bourgain, Milman, and Wolfson.

Let us say that $\mathcal{F}$ \emph{crudely factors the Hamming cubes} provided that there exist constants $c,D>0$ such that for each $n\in\nn$, there exist $F:X\to Y\in \mathcal{F}$ and $f:2^n\to X$ and constants $a,b>0$ such that for each $\ee, \delta\in 2^n$, $$\frac{a}{D} \partial(\ee, \delta) \leqslant d_X(f(\ee), f(\delta)) \leqslant a D \partial(\ee, \delta),$$ $$\frac{b}{D}\partial (\ee, \delta) \leqslant d_Y(F\circ f(\ee), F\circ f(\delta)) \leqslant bD \partial (\ee, \delta),$$ and $b\geqslant ac$. An important feature of this definition is that the scaling factors $a,b$ be uniformly equivalent (that is, $ac\leqslant b\leqslant a\lambda D^2$). Let us say that $\mathcal{F}$ \emph{factors the Hamming cubes} provided that there exists a  constant $c>0$ such that for each $D>1$ and each $n\in\nn$, there exist $F:X\to Y\in \mathcal{F}$ and $f:2^n\to X$ and constants $a,b$ such that for each $\ee, \delta\in 2^n$, $$\frac{a}{D} \partial(\ee, \delta) \leqslant d_X(f(\ee), f(\delta)) \leqslant a D \partial(\ee, \delta),$$ $$\frac{b}{D}\partial (\ee, \delta) \leqslant d_Y(F\circ f(\ee), F\circ f(\delta)) \leqslant bD \partial (\ee, \delta),$$ and $b\geqslant ac$.

We now present the main theorem.

\begin{theorem*} The following are equivalent: \begin{enumerate}[(i)]\item $\mathcal{F}$ factors the Hamming cubes. \item $\mathcal{F}$ crudely factors the Hamming cubes.  \item $\lim\sup_n a_n(\mathcal{F})>0$. \item For each $1<p<\infty$, $\lim\sup_n b_{p,n}(\mathcal{F})>0$. \item For some $1<p<\infty$, $\lim\sup_n b_{p,n}(\mathcal{F})>0$.  \end{enumerate}

\label{main theorem}

\end{theorem*}

It is obvious that $(i)\Rightarrow (ii)\Rightarrow (iii)\Rightarrow (iv)\Rightarrow (v)$. To see why $(iii)\Rightarrow (iv)$, note that for $F:X\to Y\in \mathcal{F}$ and $f:2^n\to X$, since $\varrho_X^f(\ee, d_i\ee)\leqslant \text{Lip}(f)/n$ for each $\ee\in 2^n$ and $1\leqslant i\leqslant n$,  $$\mathbb{E}\varrho_Y^f(\ee, -\ee) \leqslant \Bigl[\mathbb{E}\varrho_Y^f(\ee, -\ee)^p\Bigr]^{1/p}\leqslant b_{p,n}(\mathcal{F})\Bigl[n^{p-1}\sum_{i=1}^n \varrho_X^f(\ee, d_i\ee)^p\Bigr]^{1/p} \leqslant b_{p,n}(\mathcal{F})\text{Lip}(f),$$ so $a_n(\mathcal{F})\leqslant b_{p,n}(\mathcal{F})$.     Thus the main part  of this work is concerned with proving the implication $(v)\Rightarrow (i)$. 

We note that if $l/2\leqslant k\leqslant l$, then $b_{p,k}(\mathcal{F})\leqslant 2^{1-1/p}b_{p,l}(\mathcal{F})$. Indeed, for $F:X\to Y\in \mathcal{F}$ and $f:2^k\to X$, we can extend  $f:2^k\to X$ to a function $g:2^l\to X$ by $g(\ee)=f(\ee(1) ,\ldots, \ee(k))$.   Then $$\mathbb{E}_{2^k}\varrho^f_Y(\ee, -\ee)^p=\mathbb{E}_{2^l}\varrho_Y^g(\ee, -\ee)^p\leqslant b_{p,l}(\mathcal{F})^pl^{p-1}\sum_{i=1}^l \mathbb{E}_{2^l}\varrho_X^g(\ee, d_i\ee)^p\leqslant 2^{p-1}b_{p,l}(\mathcal{F})^pk^{p-1}\sum_{i=1}^k \mathbb{E}_{2^k}\varrho_X^f(\ee, d_i\ee)^p.$$

Applying this to  $k\in\nn$ and $l=2^{\lceil \log_2(k)\rceil}$, we deduce that $\lim\sup_n b_{p,n}(\mathcal{F})>0$ if and only if $\lim\sup_n b_{p,2^n}(\mathcal{F})>0$.  Thus our goal, completed in the fourth section of this work, will be to show that if for some $1<p<\infty$, $\lim\sup_n b_{p,2^n}(\mathcal{F})>0$, then $\mathcal{F}$ factors the Hamming cubes.  In the fifth section, we use concentration of measure to provide a quantitatively sharp proof that $(iii)\Rightarrow (i)$.

We note that the definition of our quantities $b_{p,n}(\mathcal{F})$ are reminiscent of metric type as defined by Bourgain, Milman, and Wolfson in \cite{BMW}. One may also ask about Enflo's \cite{Enflo} definition of non-linear type. In the subtype regime, however, the two notions coincide. We give the details of this in the next section.

The author wishes to thank B. Randrianantoanina for making him aware of the coarse differentiation method of Eskin, Fisher, and Whyte.

\section{Spatial versus operator results; Subtype}

We first recall a result implicitly shown in \cite{BMW} in the particular case $p=2$. The general case $1<p<\infty$ follows by substituting their Fact $2.5$ with our Lemma \ref{flat}. %Since this result is shown implicitly in the proof of Theorem $2.6$, and since it was shown in the case $p=2$, we recall the proof in the next section and indicate how to modify to hold for any $1<p<\infty$. 

\begin{theorem}\cite[Theorem $2.6$]{BMW} For $1<p<\infty$, $l\in\nn$, and $D>1$, there exists a constant $0<a<1$ such that if $(Z, d_Z)$ is a metric space and $h:2^l\to Z$ is a function such that $$\mathbb{E}d_Z(h(\ee), h(-\ee))^p > (1-a) l^{p-1}\sum_{i=1}^l \mathbb{E}d_Z(h(\ee), h(d_i\ee))^p$$ and if $$t^p=\frac{1}{l}\sum_{i=1}^l \mathbb{E} d_Z(h(\ee), h(d_i\ee))^p,$$ then for any $\ee_1, \ee_2\in 2^l$, $$\frac{1}{D}\partial (\ee_1, \ee_2)\leqslant \frac{d_Z(h(\ee_1), h(\ee_2))}{t}\leqslant D\partial (\ee_1, \ee_2).$$  

\label{BMW}

\end{theorem}

With the preceding remarkable result, in the case that $\mathcal{F}$ is a collection of identity operators (that is, in the spatial case), it is easy to complete the main theorem. This is because the function $n\mapsto b_{p,n}(\mathcal{F})$ is submultiplicative in the spatial case.  From this it follows that either $b_{p,n}(\mathcal{F})=1$ (worst possible value) for all $n\in\nn$, in which case we immediately finish by Theorem \ref{BMW}, or $b_{p,n}(\mathcal{F})\underset{n}{\to}0$.  But this method does not apply to the map case because of the lack of submultiplicativity of $n\mapsto b_{p,n}(\mathcal{F})$ in the non-spatial case.    

More generally, one is often interested in a sequence of composition submultiplicative seminorms $(T_n)_{n=1}^\infty$ defined on the class of bounded, linear operators between Banach spaces (such as Rademacher or gaussian  type $p$ \cite{Hinrichs},  Haar or marginale  type $p$ \cite{Wenzel},  or asymptotic notions of Rademacher or  basic type $p$ \cite{CDK}).  By ``composition submultiplicative,''  we mean that for any pair of operators $A,B$ such that the composition $AB$ is defined, $T_{mn}(AB)\leqslant T_n(A)T_m(B)$ for any natural numbers $m,n$.          In the case that $A=I_X$, we can apply this fact with $B=A=I_X$ to deduce that $T_{mn}(I_X)=T_{mn}(I_X^2)\leqslant T_m(I_X)T_n(I_X)$.  The standard procedure in this case is to use these inequalities to prove that either $(T_n(I_X))_{n=1}^\infty$ exhibits the quantitatively worst possible behavior for each $n$ and use this to prove the presence of certain structures (such as $\ell_1^n$ subspaces in the Rademacher case), or to prove that $(T_n(I_X))_{n=1}^\infty$ is growing/shrinking rapidly enough to ensure some non-trivial power type behavior.  This ``automatic power type'' phenomenon fails for all examples in the non-spatial case. One example, which is relevant to the subject of this work, is the diagonal operator  $F:\ell_1\to \ell_1$ given by $F\sum_{n=1}^\infty a_ne_n= \sum_{n=1}^\infty \frac{a_n}{\log (n+1)}e_n$.  This is compact, and cannot factor the Hamming cubes. But one can check that $F$ has no non-trivial Rademacher type, and therefore no non-trivial non-linear type in the sense of Bourgain, Milman, and Wolfson.  This is because for $1<p<\infty$, $$\Bigl(\mathbb{E}\|F\sum_{i=1}^n \ee(i)e_i\|^p\Big)^{1/p}\geqslant \frac{n}{\log(n+1)}$$ and $$\Bigl(\sum_{i=1}^n \|e_i\|^p\Bigr)^{1/p}= n^{1/p}=o(n/\log(n+1)).$$  More generally, we can choose any $1<p<\infty$ and a sequence $(w_n)_{n=1}^\infty$ of positive numbers vanishing as slowly as we like and define the diagonal operator $F:\ell_1\to \ell_1$ by $F\sum_{n=1}^\infty a_ne_n=\sum_{n=1}^\infty a_nw_ne_n$.   Then $b_{p,n}(F)$ is necessarily vanishing, but as slowly as we like.    Examples such as this  motivate the search for a characterization of when the worst possible behavior does not hold (in our case, worst possible behavior means factoring the Hamming cube, while in other cases it is crude finite representability/asymptotic crude finite representability of the identity operator of $\ell_1^n$, non-super weak compactness, or non-asymptotic uniform smoothability).   A technique in this case is to define a sequence of \emph{subtype} constants of the map (or family of maps). One instance of this approach is due to Beauzamy \cite{Beauzamy}, who gave a characterization of when the identity on $\ell_1$ is crudely finitely representable in a linear operator using a sequence of constants which are the linear analogues of our $a_n(\mathcal{F})$. Hinrichs proved a similar result using constants which are the linear analogues of our $b_{2,n}(\mathcal{F})$. In \cite{CDK}, asymptotic analogues of the results of Beauzamy and Hinrichs were proven for both the asymptotic linear analogues of the $a_n(\mathcal{F})$ and $b_{p,n}(\mathcal{F})$ constants.

The general approach to subtype problems is as follows: Suppose we have a sequence  $(T_n)_{n=1}^\infty$ as in the previous paragraph and positive numbers $(c_n)_{n=1}^\infty$ such that for each $\lambda>0$, $\lambda c_n$ is the supremum of $T_n(A)$ as $A$ ranges over all bounded, linear operators with $\|A\|\leqslant \lambda$.    Then one may ask if, for a given class $\mathcal{A}$ of operators with norms not more than $\lambda$, does $(\sup_{A\in \mathcal{A}} T_n(A))_{n=1}^\infty$ exhibit the essentially worst possible behavior with respect to the sequence $(c_n)_{n=1}^\infty$ (that is, $\lim\sup_n \sup_{A\in \mathcal{A}} T_n(A)/c_n>0$)?  We then say $A$ has \emph{subtype} if it does not exhibit the  worst possible behavior (that is, $\lim_n \sup_{A\in \mathcal{A}} T_n(A)/c_n=0$).  This has been applied when $T_n$ is the Rademacher/gaussian/Haar/martingale type $p$ norms. More generally, we may isolate non-linear subtype properties by replacing continuous, linear operators with Lipschitz functions, replacing the $(T_n)_{n=1}^\infty$ sequence with a sequence $(\tau_n)_{n=1}^\infty$ defined on the class of Lipschitz maps between metric spaces, and by replacing operator norm with Lipschitz constant.  One  then says that a class $\mathcal{F}$ has subtype if $\lim_n \sup_{F\in \mathcal{F}}\tau_n(F)/c_n=0$. This is the approach we take.

Now for a family $\mathcal{F}$ of $\lambda$-Lipschitz maps, let us define $e_{p,n}(\mathcal{F})$ to be the smallest constant $t>0$ such that for any $F:X\to Y\in \mathcal{F}$ and  $f:2^n\to X$, $$\mathbb{E}\varrho^f_Y(\ee, -\ee)^p \leqslant t^p \sum_{i=1}^n \mathbb{E} \varrho_X^f(\ee, d_i\ee)^p.$$   Note that  $e_{p,n}(F)\leqslant \lambda n^{1-1/p}$.  Therefore with $c_n=n^{1-1/p}$, we can say $\mathcal{F}$ has \emph{Enflo subtype} if $\lim_n e_{p,n}(\mathcal{F})/c_n=0$. But $e_{p,n}(\mathcal{F})/n^{1-1/p}=b_{p,n}(\mathcal{F})$.  Therefore the subtype approach applied to Enflo type recovers the same condition as the Bourgain, Milman, Wolfson approach.

\section{Rigidity results}

\begin{lemma} For $1<p<\infty$, $n\in\nn$, and $\Phi>1$, there exists $\phi=\phi(p,n,\Phi)\in (0,1)$ such that if $a=(a_i)_{i=1}^n\in \ell_p^n$ satisfies $\|a\|_{\ell_1^n}^p>\phi n^{p-1}\|a\|_{\ell_p^n}^p$, then $\max_i |a_i|\leqslant \Phi \min_i |a_i|$.

\label{flat}
\end{lemma}

\begin{proof} By the uniform convexity of $\ell_p^n$, there exists $0<\delta<1$ such that if $x,y\in B_{\ell_p^n}$ are such that $\|x+y\|_{\ell_p^n}>2(1-\delta)$, then $\|x-y\|_{\ell_p^n}<\frac{1}{2n^{1/p}}\cdot \frac{\Phi-1}{\Phi}$. Now let $\phi=(1-\delta)^p$.    Suppose $a=(a_i)_{i=1}^n\in \ell_p^n$ satisfies $\|a\|_{\ell_1^n}^p>\phi n^{p-1} \|a\|_{\ell_p^n}^p$.  Without loss of generality, let us assume that $a_i=|a_i|$ for all $1\leqslant i\leqslant n$.    Let $x_i=a_i/\|a\|_{\ell_p^n}$ and $x=(x_i)_{i=1}^n\in S_{\ell_p^n}$. Let $y_i=n^{-1/p}$ for $1\leqslant i\leqslant n$ and $y=(y_i)_{i=1}^n\in S_{\ell_p^n}$. Note that $\|x\|_{\ell_1^n}>\phi^{1/p} n^{1-1/p}=(1-\delta)n^{1-1/p}$ and $\|y\|_{\ell_1^n}=n^{1-1/p}$. Then $$\|x+y\|_{\ell_p^n} \geqslant \|x+y\|_{\ell_1^n}/n^{1-1/p} = \frac{\|x\|_{\ell_1^n}+\|y\|_{\ell_1^n}}{n^{1-1/p}}\geqslant (1-\delta)+1>2(1-\delta).$$  Therefore $\|x-y\|_{\ell_p^n}\leqslant \frac{1}{2n^{1/p}}\cdot\frac{\Phi-1}{\Phi}$.   Since $\max_i x_i\geqslant 1/n^{1/p}$, we deduce  that $$\max_i x_i-\min_i x_i \leqslant  |n^{-1/p}-\max_i x_i|+|n^{-1/p}-\min_i x_i| \leqslant 2\|x-y\|_{\ell_\infty^n}\leqslant \frac{1}{n^{1/p}}\cdot \frac{\Phi-1}{\Phi} \leqslant \Bigl(\frac{\Phi-1}{\Phi}\Bigr)\max_i x_i.$$  Rearranging yields that $$\max_i x_i\leqslant \Phi \min_i x_i,$$ and we deduce the result by homogeneity.

\end{proof}

\begin{lemma} Fix $1<p<\infty$.  Let $\Omega$ be a  probability space and let $D_Y, D_X, E_Y, E_X:\Omega\to \rr$, $\lambda, \Theta>0$, $a,b,\nu,  \mu\in (0,1)$ be such that \begin{enumerate}[(i)]\item $D_Y, D_X, E_Y, E_X$ are non-negative, measurable functions on $\Omega$  such that  $D_Y\leqslant E_Y$,  $D_X\leqslant E_X$,  $D_Y\leqslant \lambda^p D_X$,   $E_Y\leqslant \lambda^p E_X$,   $D_Y\leqslant (1+\nu)\Theta^p E_X$, \item $\mathbb{E}D_Y>(1-\nu)\Theta^p \mathbb{E}E_X$,  $\mathbb{E}D_Y>(1-\mu) \mathbb{E}E_Y$,  $\mathbb{E}D_X>(1-\mu)\mathbb{E}E_X$, \item $\lambda^p\Bigl(\frac{2\mu}{a}+\frac{2\nu}{b}\Bigr)<\Theta^p(1-\nu)$. \end{enumerate}

Then there exists $\ee\in \Omega$ such that $D_Y(\ee)>(1-a)E_Y(\ee)$,  $D_X(\ee)>(1-a)E_Y(\ee)$, and $D_Y(\ee)>(1-b)\Theta^p E_X(\ee)$. 

\label{expect}
\end{lemma}

\begin{proof} Let $A_Y=(D_Y\leqslant (1-a)E_Y),$ $A_X=(D_X\leqslant (1-a)E_X),$ and $B=(D_Y\leqslant (1-b)\Theta^p E_X).$   Then the conclusion  of the lemma is equivalent to $A^c_Y\cap A_X^c\cap B^c\neq \varnothing$. We work by contradiction. Assume $A^c_Y\cap A^c_X\cap B^c=\varnothing$, so $\Omega=A_Y\cup A_X\cup B$.    Let us first note that $$ (1-\mu)\mathbb{E}E_Y  < \mathbb{E}D_Y = \mathbb{E}1_{A_Y}D_Y + \mathbb{E}1_{A_Y^c}D_Y \leqslant (1-a)\mathbb{E}1_{A_Y}E_Y + \mathbb{E}1_{A_Y^c}E_Y = \mathbb{E}E_Y- a \mathbb{E}1_{A_Y}E_Y.$$ From this it follows that $$ \mathbb{E}1_{A_Y}E_Y \leqslant \frac{\mu}{a}\mathbb{E}E_Y. $$    By replacing each $Y$ with $X$, we deduce that $$ \mathbb{E}1_{A_X}E_X \leqslant \frac{\mu}{a}\mathbb{E}E_X. $$

Also, \begin{align*} \Theta^p (1-\nu)\mathbb{E}E_X & < \mathbb{E}D_Y = \mathbb{E}1_BD_Y +\mathbb{E}1_{B^c}D_Y \leqslant \Theta^p(1-b)\mathbb{E}1_B E_X + \Theta^p(1+\nu)\mathbb{E}1_{B^c}E_X. \end{align*}  Dividing by $\Theta^p$ and rearranging yields that  $$\mathbb{E}1_B E_X \leqslant \frac{2\nu}{b}\mathbb{E}E_X.$$

Recalling that $A_Y\cup A_X\cup B=\Omega$,  we deduce that \begin{align*} \Theta^p(1-\nu) \mathbb{E}E_X & < \mathbb{E}D_Y \leqslant \mathbb{E}1_{A_Y}D_Y+\mathbb{E}1_{A_X}D_Y+ \mathbb{E}1_BD_Y \\ & \leqslant \mathbb{E}1_{A_Y}E_Y + \lambda^p \mathbb{E}1_{A_X}E_X +\lambda^p \mathbb{E}1_BE_X \\ & \leqslant \frac{\mu}{a}\mathbb{E}E_Y + \frac{\mu\lambda^p}{a} \mathbb{E}E_X +\frac{2\nu\lambda^p}{b} \mathbb{E}E_X \\ & \leqslant \lambda^p\Bigl(\frac{2\mu}{a}+\frac{2\nu}{b}\Bigr)\mathbb{E}E_X. \end{align*} 

Since $\mathbb{E}E_X>0$, this contradicts $(iii)$ and finishes the proof.

\end{proof}

For a natural number $n$, we let $[n]=\{1, \ldots, n\}$ denote the integer interval.   Fix natural numbers $ l_1, \ldots, l_d$ and let $L=\prod_{j=1}^d l_j$. We define $T=\cup_{i=0}^d \Lambda_i$  as follows.   We let $\Lambda_0=\{[L]\}$ consist of a single integer interval.    Now suppose that for $i<d$, $\Lambda_i$ has been defined and consists of pairwise disjoint subintervals of $[L]$ each of which has cardinality $\prod_{j=i+1}^d l_j$.    For each $I\in \Lambda_i$, let $\mathcal{I}_I=\{J^I_1, \ldots, J^I_{l_{i+1}}\}$ be a partition of $I$ into subintervals of equal cardinalilty (and therefore of cardinality $\prod_{j=i+2}^d l_j$).    Now let $\Lambda_{i+1}=\cup_{I\in \Lambda_i} \mathcal{I}_I$.   This completes the recursive definition of $\Lambda_0, \ldots, \Lambda_d$. Now let $T=\cup_{i=0}^d \Lambda_i$.   We refer to $T$ as the $(l_1, \ldots, l_d)$ \emph{interval tree}.  For $0<j\leqslant d$ and $J\in \Lambda_j$, let $J^-$ be the member $I$ of $\Lambda_{j-1}$ such that $J\subset I$.   That is, $J^-$ is the interval $I\in \Lambda_j$ such that $J\in \mathcal{I}_I$.   

\begin{rem}\upshape Suppose $l_1, \ldots, l_{d+1}$ are natural numbers and $T$ is the $(l_1, \ldots, l_{d+1})$ interval tree.  Suppose that $(t_I)_{I\in T}$ is a collection of non-negative numbers such that for each $0\leqslant j\leqslant d$ and $I\in \Lambda_j$, $t_I\leqslant \sum_{J^-=I} t_J$.   Then using this fact repeatedly yields that for any $0\leqslant i<j\leqslant d+1$ and $I\in \Lambda_i$, $$t_I\leqslant \sum_{I\supset J\in \Lambda_j} t_J.$$

Also, by H\"{o}lder's inequality, it follows that for any such $j$ and $I$, $$t_I^p\leqslant l_{j+1}^{p-1}\sum_{J^-=I} t^p_J,$$ and more generally, $$t^p_I \leqslant \Bigl(\prod_{m=i+1}^j l_m^{p-1}\Bigr)\sum_{I\supset J\in \Lambda_j} t^p_J$$ for any $0\leqslant i<j\leqslant d+1$ and $I\in \Lambda_i$.   We will use this fact frequently in this section. 

\label{vase}

\end{rem}

\begin{lemma} Fix $1<p<\infty$. Fix natural numbers $l_1, \ldots, l_d$,  $0<\mu<1$,  $\lambda, \Theta>0$, and $M>\lambda/\Theta$. Then for any $0<\eta_1<1$,  there exists $0<\eta<\eta_1$ with the following property: Suppose $l_{d+1}$ is a natural number, $T$ is the $(l_1, \ldots, l_{d+1})$ interval tree, and $(r_I)_{I\in T}, (s_I)_{I\in T}$ are non-negative numbers such that \begin{enumerate}[(i)]\item for each $I\in T$, $r_I\leqslant \lambda s_I$, \item for each $I\in T\setminus \Lambda_{d+1}$, $r_I\leqslant \sum_{J^-=I}r_J$ and  $s_I\leqslant \sum_{J^-=I}s_J$, \item for each $I\in \Lambda_d$, $r_I^p\leqslant (1+\eta)\Theta^p l_{d+1}^{p-1} \sum_{J^-=I}s^p_J$, \item $r_{[L]}>(1-\eta)\Theta^p\Bigl(\prod_{i=1}^{d+1} l_i^{p-1}\Bigr) \sum_{I\in \Lambda_{d+1}}s_I^p$. \end{enumerate}   Then for any $0\leqslant j\leqslant d$, $0\leqslant i<d$, and $I_1\in \Lambda_i$, $\max_{I\in \Lambda_j} s_I\leqslant M\min_{I\in \Lambda_j} s_I$ and $r_{I_1}^p>(1-\mu)l_{j+1}^{p-1}\sum_{J^-=I_1} r_J^p$.

 \label{eta}

\end{lemma}

\begin{proof} First fix $ \Phi>1$ such that $M>\Phi^3 \lambda/\Theta$.   Now let $0<\phi<1$ be such that for any $1\leqslant n\leqslant \prod_{i=1}^d l_i$ and any $v=(v_i)_{i=1}^n\in \ell_p^n$ with $\max_i |v_i|>\Phi\min_i |v_i|$, $\|v\|_{\ell_1^n}^p<\phi n^{p-1}\|v\|_{\ell_p^n}^p$.  Such a $\phi$ exists by Lemma \ref{flat}.  Now fix $0<\eta<\eta_1$ so small that \begin{enumerate}[(a)]\item $\phi(1+\eta)<1-\eta$, \item $\frac{\Phi^p}{1+\eta}-\frac{1}{1-\eta}>\Bigl(\frac{1}{1-\eta}-\frac{1}{1+\eta}\Bigr)\Bigl(\prod_{i=1}^d l_i\Bigr)\Phi^p$, \item $M> \frac{ \Phi^3 \lambda}{(1-\eta)^{1/p}\Theta}$, \item $\Bigl(1-\frac{\mu}{\Phi^p \prod_{i=1}^d l_i}\Bigr)(1+\eta)<(1-\eta)$.  \end{enumerate}

Now suppose that $l_{d+1}$, $(r_I)_{I\in T}$, $(s_I)_{I\in T}$ are as in the lemma.

Step $1$: For any $0\leqslant j\leqslant d$, $\max_I r_I\leqslant \Phi \min_I r_I$.

If it were not so, then by the choice of $\phi$ applied to the vector $(r_I)_{I\in \Lambda_j}\in \ell_p(\Lambda_j)$, $$\Bigl(\sum_{I\in \Lambda_j} r_I\Bigr)^p \leqslant \phi\Bigl(\prod_{i=1}^j l_i^{p-1}\Bigr)\sum_{I\in \Lambda_j}r_I^p.$$   Then \begin{align*} r_{[L]}^p & \leqslant \Bigl(\sum_{I\in \Lambda_j} r_I\Bigr)^p < \phi\Bigl(\prod_{i=1}^j l_i^{p-1}\Bigr)\sum_{I\in \Lambda_j}r_I^p  \leqslant \phi\Bigl(\prod_{i=1}^d l_i^{p-1}\Bigr)\sum_{I\in \Lambda_d} r_I^p \\  & \leqslant \phi(1+\eta)\Theta^p\Bigl(\prod_{i=1}^{d+1} l_i^{p-1}\Bigr)\sum_{I\in \Lambda_{d+1}}s_I^p.\end{align*}  But since $r_{[L]}^p>(1-\eta)\Theta^p (\prod_{i=1}^{d+1}l_i^{p-1})\sum_{I\in \Lambda_{d+1}}s_I^p$, we contradict item $(a)$ of our choice of $\eta$. This completes Step $1$.  Note that this implies that for each $0\leqslant j\leqslant d$ and $I\in \Lambda_j$, $r_I>0$.

Step $2$: For any $0\leqslant j\leqslant d$, $$\max_{I\in \Lambda_j} \sum_{I\supset J\in \Lambda_{d+1}}s^p_J \leqslant \Phi^{2p} \min_{I\in \Lambda_j} \sum_{I\supset J\in \Lambda_{d+1}}s^p_J.$$

If it were not so, we could find $0\leqslant j\leqslant d$ and $I_1, I_2\in \Lambda_j$ such that $$\frac{1}{ \Phi^{2p}}\sum_{I_1\supset J\in \Lambda_{d+1}}s^p_J>\sum_{I_2\supset J\in \Lambda_{d+1}}s^p_J.$$   Then by Step $1$,  \begin{align*} r_{I_1}^p & \leqslant \Phi^p r_{I_2}^p \leqslant \Phi^p \Bigl(\prod_{i=j+1}^d l_i^{p-1}\Bigr)\sum_{I_2\supset J\in \Lambda_d} r^p_J \leqslant \Phi^p(1+\eta)\Theta^p \Bigl(\prod_{i=j+1}^{d+1} l_i^{p-1}\Bigr)\sum_{I_2\supset J\in \Lambda_{d+1}}s^p_J \\ & \leqslant \frac{(1+\eta)\Theta^p}{\Phi^p}\Bigl(\prod_{i=j+1}^{d+1} l_i^{p-1}\Bigr)\sum_{I_1\supset J\in \Lambda_{d+1}}s_J^p.  \end{align*} Note that for any $I\in \Lambda_j$, $$r_I^p\leqslant \Bigl(\prod_{i=j+1}^d l_i^{p-1}\Bigr)\sum_{I\supset J\in \Lambda_d} r_J^p \leqslant (1+\eta)\Theta^p\Bigl(\prod_{i=j+1}^{d+1} l_i^{p-1}\Bigr)\sum_{I\in \Lambda_{d+1}} s_I^p.$$   Since $$(1-\eta)\Theta^p\Bigl(\prod_{i=1}^{d+1} l_i^{p-1}\Bigr)\sum_{I\in \Lambda_{d+1}}s_I^p <r^p_{[L]} \leqslant \Bigl(\prod_{i=1}^j l_i^{p-1}\Bigr)\sum_{I\in \Lambda_j} r_I^p,$$  we see that \begin{align*} \frac{1}{(1+\eta)\Theta^p}\Bigl[ \Phi^p r_{I_1}^p +\sum_{I_1\neq I\in \Lambda_j}r_I^p\Bigr] \leqslant \Bigl(\prod_{i=j+1}^{d+1} l_i^{p-1}\Bigr)\sum_{I\in \Lambda_{d+1}}s^p_I \leqslant \frac{1}{(1-\eta)\Theta^p}\sum_{I\in \Lambda_j}r^p_I. \end{align*}   Manipulating the first and last terms of this inequality, we deduce that \begin{align*} \Bigl(\frac{\Phi^p}{1+\eta}-\frac{1}{1-\eta}\Bigr)r_{I_1}^p & \leqslant \Bigl(\frac{1}{1-\eta}-\frac{1}{1+\eta}\Bigr)\sum_{I_1\neq I\in \Lambda_j} r_I^p  \leqslant \Bigl(\frac{1}{1-\eta}-\frac{1}{1+\eta}\Bigr)\Phi^p |\Lambda_j|r_{I_1}^p \\ & \leqslant \Bigl(\frac{1}{1-\eta}-\frac{1}{1+\eta}\Bigr)\Phi^p \Bigl(\prod_{i=1}^d l_i\Bigr)r_{I_1}^p. \end{align*} Since $r_{I_1}>0$, we reach a contradiction of $(b)$ of our choice of $\eta$.

Step $3$: For any $0\leqslant j\leqslant d$, $\max_{I\in \Lambda_j} s_I\leqslant M    \min_{I\in \Lambda_j} s_I$.   Fix such a $j$ and let  $$R=\max_{I\in \Lambda_j} r_I \hspace{15mm}  S=\max_{I\in \Lambda_j} s_I \hspace{15mm} S_1=\max_{I\in \Lambda_j}\Bigl(\prod_{i=j+1}^{d+1}l_i^{p-1}\Bigr)\sum_{I\supset J\in \Lambda_{d+1}}s^p_J$$ and $$r=\min_{I\in \Lambda_j} r_I \hspace{15mm}  s=\min_{I\in \Lambda_j} s_I \hspace{15mm} s_1=\min_{I\in \Lambda_j}\Bigl(\prod_{i=j+1}^{d+1}l_i^{p-1}\Bigr)\sum_{I\supset J\in \Lambda_{d+1}}s^p_J. $$ We know from Step $1$ that $R\leqslant \Phi r$. We know from Step $2$ that $S_1\leqslant \Phi^{2p} s_1$. We know from hypothesis that $R\leqslant \lambda S$, and we know from Remark \ref{vase} that $S^p\leqslant S_1$.   Moreover, \begin{align*} \Bigl(\prod_{i=1}^j l_i^{p-1}\Bigr)s_1|\Lambda_j| & \leqslant \Bigl(\prod_{i=1}^{d+1}l_i^{p-1}\Bigr)\sum_{I\in \Lambda_j}\sum_{I\supset J\in \Lambda_{d+1}}s^p_J =\Bigl(\prod_{i=1}^{d+1}l_i^{p-1}\Bigr)\sum_{I\in \Lambda_{d+1}}s_I^p < [(1-\eta)\Theta^p]^{-1} r_{[L]}^p \\ & \leqslant [(1-\eta)\Theta^p]^{-1} \Bigl(\prod_{i=1}^j l_i^{p-1}\Bigr) \sum_{I\in \Lambda_j} r_I^p \leqslant [(1-\eta)\Theta^p]^{-1} \Bigl(\prod_{i=1}^j l_i^{p-1}\Bigr)R^p |\Lambda_j| \\ & \leqslant  [(1-\eta)\Theta^p]^{-1}\Phi^p \Bigl(\prod_{i=1}^j l_i^{p-1}\Bigr)r^p |\Lambda_j| \leqslant [(1-\eta)\Theta^p]^{-1}\Phi^p \lambda^p \Bigl(\prod_{i=1}^j l_i^{p-1}\Bigr)s^p |\Lambda_j|.\end{align*}     Therefore $ s_1\leqslant [(1-\eta)\Theta^p]^{-1}\Phi^p \lambda^p s^p$.  Also, $S^p\leqslant S_1\leqslant \Phi^{2p} s_1,$ so $$S^p \leqslant \Phi^{2p} s_1 \leqslant [(1-\eta)\Theta^p]^{-1} \Phi^{3p} \lambda^p s^p.$$ Taking $p^{th}$ roots and appealing to $(c)$ finishes Step $3$.

Step $4$: For any $0\leqslant j<d$ and $I\in \Lambda_j$, $r^p_I>(1-\mu)l_{j+1}^{p-1}\sum_{J^-=I}r^p_J$.    If it were not so, then for some $I_1\in \Lambda_j$, $r^p_I\leqslant (1-\mu)l_{j+1}^{p-1}\sum_{J^-=I_1}r_J^p$.  Let us note that $$\sum_{J^-=I_1} r^p_J \geqslant \frac{1}{\Phi^p |\Lambda_{j+1}|}\sum_{J\in \Lambda_{j+1}}r^p_J \geqslant \frac{1}{\Phi^p \prod_{i=1}^d l_i}\sum_{J\in \Lambda_{j+1}}r^p_I,$$ so \begin{align*} \sum_{I\in \Lambda_j} r_I^p & \leqslant (1-\mu)l_{j+1}^{p-1}\sum_{J^-=I_1}r^p_J + l_{j+1}^{p-1}\sum_{I_1\neq I\in \Lambda_j}\sum_{ J^-=I}r^p_J \leqslant l_{j+1}^{p-1}\sum_{J\in \Lambda_{j+1}}r^p_J-l_{j+1}^{p-1}\mu \sum_{J^-=I_1}r_J^p \\ & \leqslant \Bigl(1-\frac{\mu}{\Phi^p\prod_{i=1}^d l_i}\Bigr)l_{j+1}^{p-1}\sum_{J\in \Lambda_{j+1}}r^p_J.  \end{align*}  Then \begin{align*} r_{[L]}^p & \leqslant \Bigl(\prod_{i=1}^j r_i^{p-1}\Bigr)\sum_{i\in \Lambda_j}r_I^p \leqslant \Bigl(1-\frac{\mu}{\Phi^p\prod_{i=1}^d l_i}\Bigr)\Bigl(\prod_{i=1}^{j+1}l_i^{p-1}\Bigr)\sum_{I\in \Lambda_{j+1}}r_I^p \\ & \leqslant \Bigl(1-\frac{\mu}{\Phi^p\prod_{i=1}^d l_i}\Bigr)\Bigl(\prod_{i=1}^d l_i^{p-1}\Bigr)\sum_{I\in \Lambda_d}r_I^p \\ & \leqslant \Bigl(1-\frac{\mu}{\Phi^p\prod_{i=1}^d l_i}\Bigr)(1+\eta)\Theta^p \Bigl(\prod_{i=1}^{d+1}l_i^{p-1}\Bigr)\sum_{I\in \Lambda_{d+1}}s_I^p.\end{align*}    But since $r_{[L]}^p> (1-\eta)\Theta^p \Bigl(\prod_{i=1}^{d+1}l_i^{p-1}\Bigr)\sum_{I\in \Lambda_{d+1}}s^p_I$, this contradicts $(d)$ and finishes the proof.

\end{proof}

The following result is similar in spirit to the coarse differentiation result of Eskin, Fisher, and Whyte \cite{EFW}.

\begin{lemma} Fix $1< p<\infty$. Fix natural numbers $(l_1, \ldots, l_{d+1})$ and let $T$ be the $(l_1, \ldots, l_{d+1})$ interval tree.     Suppose $0<\Delta,\mu<1$, $M>1$, $m\in\nn$,  $\lambda>0$, and $(s_I)_{I\in T}\subset (0,\infty)$ are such that \begin{enumerate}[(i)]\item for each $I\in T\setminus \Lambda_{d+1}$, $s_I\leqslant \sum_{J^-=I}s_J$, \item for all $0\leqslant j\leqslant d$, $\max_{I\in \Lambda_j}s_I\leqslant M\min_{I\in \Lambda_j}s_I$, \item $s_{[L]}^p>(1-\nu/2)\frac{\Theta^p}{\lambda^p}\Bigl(\prod_{i=1}^{d+1}l_i^{p-1}\Bigr)\sum_{I\in \Lambda_{d+1}}s_I^p$, \item $(1-\mu\Delta/M^p)^m<(1-\nu/2)\frac{\Theta^p}{\lambda^p}$. \end{enumerate}

For each $0\leqslant j<d$, let $I_j=\{I\in \Lambda_j: s^p_I \leqslant (1-\mu)l_{j+1}^{p-1}\sum_{J^-=I} s^p_J\}$ and let $B=\{j<d: |I_j|\geqslant \Delta |\Lambda_j|\}$.   Then $|B| \leqslant m$. 

\label{goodX}

\end{lemma}

\begin{proof} First suppose $j\in B$.    Let $s=\min_{I\in \Lambda_{j+1}} s_I$ and $S=\max_{I\in \Lambda_{j+1}}s_I\leqslant sM$. Let $A=\{J\in \Lambda_{j+1}:J^-\in I_j\}$ and note that $|A|/|\Lambda_{j+1}|=|I_j|/|\Lambda_j| \geqslant \Delta$.   Then \begin{align*} \sum_{I\in \Lambda_j} s_I^p & = \sum_{I\in I_j}s^p_I+\sum_{I\in \Lambda_j\setminus I_j} s^p_I \leqslant (1-\mu)l_{j+1}^{p-1}\sum_{I\in I_j}\sum_{J^-=I}s_J^p + l_{j+1}^{p-1}\sum_{I\in \Lambda_j\setminus I_j}\sum_{J^-=I}s^p_J \\ & = l_{j+1}^{p-1}\sum_{I\in \Lambda_{j+1}}s_I^p -\mu l_{j+1}^{p-1} \sum_{I\in A} s^p_I \leqslant l_{j+1}^{p-1}\sum_{I\in \Lambda_{j+1}}s^p_I - \mu l_{j+1}^{p-1} s^p|A| \leqslant l_{j+1}^{p-1}\sum_{I\in \Lambda_{j+1}}s^p_I -l_{j+1}^{p-1} \frac{\mu S^p}{M^p}|A| \\ & \leqslant l_{j+1}^{p-1}\sum_{I\in \Lambda_{j+1}}s^p_I - \frac{\mu S^p}{M^p}l_{j+1}^{p-1} \Delta |\Lambda_{j+1}| \leqslant l_{j+1}^{p-1}\sum_{I\in \Lambda_{j+1}}s^p_I - \frac{\mu \Delta}{M^p}l_{j+1}^{p-1}\sum_{I\in \Lambda_{j+1}}s^p_I \\ & = (1-\mu\Delta/M^p)l_{j+1}^{p-1}\sum_{I\in \Lambda_{j+1}}s_I^p. \end{align*}

Now suppose that $|B|>m$ and fix $0\leqslant j_0<\ldots <j_m$, $j_i\in B$ for each $0\leqslant i\leqslant m$.  Then \begin{align*} s^p_{[L]} & \leqslant\Bigl(\prod_{i=1}^{j_0} l_i^{p-1}\Bigr) \sum_{I\in \Lambda_{j_0}}s^p_I \leqslant \Bigl(1-\frac{\mu\Delta}{M^p}\Bigr)\Bigl(\prod_{i=1}^{j_0+1} l_i^{p-1}\Bigr) \sum_{I\in \Lambda_{j_0+1}}s^p_I \\ & \leqslant \Bigl(1-\frac{\mu\Delta}{M^p}\Bigr)\Bigl(\prod_{i=1}^{j_1} l_i^{p-1}\Bigr) \sum_{I\in \Lambda_{j_1}}s^p_I  \leqslant  \Bigl(1-\frac{\mu\Delta}{M^p}\Bigr)^2 \Bigl(\prod_{i=1}^{j_1+1} l_i^{p-1}\Bigr) \sum_{I\in \Lambda_{j_1+1}}s^p_I\\ & \leqslant \ldots \leqslant \Bigl(1-\frac{\mu\Delta}{M^p}\Bigr)^m \Bigl(\prod_{i=1}^{j_m+1}l_i^{p-1}\Bigr)\sum_{I\in \Lambda_{j_m+1}} s_I^p \\ & \leqslant \Bigl(1-\frac{\mu\Delta}{M^p}\Bigr)^m \Bigl(\prod_{i=1}^{d+1}l_i^{p-1}\Bigr)\sum_{I\in \Lambda_{d+1}} s_I^p. \end{align*} This contradicts $(iii)$ and $(iv)$ and finishes the proof.

\end{proof}

\begin{lemma} Fix $1< p<\infty$. Fix natural numbers $(l_1, \ldots, l_{d+1})$ and let $T$ be the $(l_1, \ldots, l_{d+1})$ interval tree.     Suppose $0<\Delta, \nu<1$, $\lambda>0$,  $M>1$,   and $(r_I)_{I\in T}, (s_I)_{I\in T}\subset [0,\infty)$ are such that \begin{enumerate}[(i)]\item for each $I\in T$, $r_I\leqslant \lambda s_I$, \item for each $I\in T\setminus \Lambda_{d+1}$, $r_I\leqslant \sum_{J^-=I}r_J$,  $s_I\leqslant \sum_{J^-=I}s_J$,   \item for all $0\leqslant j\leqslant d$, $\max_{I\in \Lambda_j} s_I \leqslant M\min_{I\in \Lambda_j} s_I$, \item $\Delta M^p \leqslant \nu\Theta^p/2\lambda^p$, \item $r_{[L]}^p>(1-\nu/2)\Theta^p\Bigl(\prod_{i=1}^{d+1}l_i^{p-1}\Bigr)\sum_{I\in \Lambda_{d+1}}s^p_I$.  \end{enumerate}

Then for any $0\leqslant j<d$, $|\{I\in \Lambda_j: r_I^p\leqslant (1-\nu)l_{j+1}^{p-1}\Theta^p \sum_{J^-=I}s^p_I\}|<(1-\Delta)|\Lambda_j|$. 

\label{YtoX}

\end{lemma}

\begin{proof} Suppose not. Fix $0\leqslant j<d$ such that $B=\{I\in \Lambda_j: r_I^p\leqslant (1-\nu)\Theta^p\sum_{J^-=I}s^p_I\}$ has cardinality at least $(1-\Delta)|\Lambda_j|$.  Let $s=\min_{I\in \Lambda_{j+1}}s_I$ and $S=\max_{I\in \Lambda_{j+1}}s_I\leqslant sM$.  Let $A=\{J\in \Lambda_{j+1}: J^-\in B\}$ and note that $|A|/|\Lambda_{j+1}|=|B|/|\Lambda_j|\geqslant 1-\Delta$.  Now \begin{align*} \sum_{J\in \Lambda_{j+1}\setminus A}s^p_J \leqslant |\Lambda_{j+1}\setminus A|S^p\leqslant \Delta |\Lambda_{j+1}|S^p \leqslant \Delta M^p\sum_{I\in \Lambda_{j+1}}s_I^p \leqslant \frac{\nu \Theta^p}{2\lambda^p}\sum_{I\in \Lambda_{j+1}}s_I^p. \end{align*} Then \begin{align*} r_{[L]}^p & \leqslant \Bigl(\sum_{i=1}^j l_i^{p-1}\Bigr)\sum_{I\in \Lambda_j} r^p_I = \Bigl(\sum_{i=1}^j l_i^{p-1}\Bigr)\Bigl[\sum_{I\in B} r^p_I+\sum_{I\in \Lambda_j\setminus B} r^p_I\Bigr] \\ & \leqslant \Bigl(\sum_{i=1}^{j+1} l_i^{p-1}\Bigr)\Bigl[(1-\nu)\Theta^p \sum_{I\in B}\sum_{J^-=I}s^p_J+\lambda^p\sum_{I\in \Lambda_j\setminus B}\sum_{J^-=I}s^p_J\Bigr] \\ & = \Bigl(\sum_{i=1}^{j+1} l_i^{p-1}\Bigr)\Bigl[(1-\nu)\Theta^p \sum_{I\in A}s^p_I+\lambda^p\sum_{I\in \Lambda_{j+1}\setminus A}s^p_I\Bigr] \\ & \leqslant \Bigl(\prod_{i=1}^{j+1}l_i^{p-1}\Bigr)\Bigl[(1-\nu)\Theta^p\sum_{I\in \Lambda_{j+1}}s_I^p+ \frac{\nu\Theta^p}{2}\sum_{I\in \Lambda_{j+1}}s_I^p \Bigr] \\ & = (1-\nu/2)\Theta^p \Bigl(\prod_{i=1}^{j+1}l_i^{p-1}\Bigr) \sum_{I\in \Lambda_{j+1}}s_I^p  \\ & \leqslant (1-\nu/2)\Theta^p \Bigl(\prod_{i=1}^{d+1}l_i^{p-1}\Bigr)\sum_{I\in \Lambda_{d+1}}s^p_I. \end{align*} This contradiction finishes the proof.

\end{proof}

\section{Proof of $(v)\Rightarrow (i)$}

Fix $1<p<\infty$. Let $\lambda=\sup_{F\in \mathcal{F}}\text{Lip}(F)\in (0,\infty)$ and let $\Theta=\lim\sup_n b_{p, 2^n}(\mathcal{F})\in (0,\lambda]$.   Fix $0<\vartheta<\Theta$.

Fix $0<b<1$ such that $(1-b)^{1/p}\Theta>\vartheta$. Next fix $0<\nu<b$ such that $\nu/b<(1-\nu)\Theta^p/4\lambda^p$. Fix $n_0\in \nn$ such that for all $n\geqslant n_0$, $b_{p, 2^n}(\mathcal{F})\leqslant (1+\nu)^{1/p}\Theta$. Fix $n\geqslant n_0$ and $D>1$, and let $l=2^n$.    By Theorem \ref{BMW}, there exists $0<a<1$ such that if $(Z, d_Z)$ is any metric space and $h:2^l\to Z$ satisfies $$\mathbb{E}d_Z(h(\ee), h(-\ee))^p>(1-a)l^{p-1}\sum_{i=1}^l d_Z(h(\ee), h(d_i\ee))^p$$ and if $$t^p=\frac{1}{l}\sum_{i=1}^l d_Z(h(\ee), h(d_i\ee))^p,$$ then for any $\ee_1, \ee_2\in 2^l$, $$\frac{1}{D}\partial(\ee_1, \ee_2) \leqslant \frac{d_Z(h(\ee_1), h(\ee_2))}{t}\leqslant D\partial(\ee_1, \ee_2).$$   Now fix $0<\mu<a$ such that $\mu/a<(1-\nu)\Theta^p/4\lambda^p$.    Fix $M>\lambda/\Theta$.  Fix $0<\Delta<1$ such that $\Delta M^p<\nu\Theta^p/2\lambda^p$.    Fix $m\in\nn$ such that $$\Bigl(1-\frac{\Delta \mu}{M^p}\Bigr)^m<(1-\nu/2)\frac{\Theta^p}{\lambda^p}.$$    Fix $d>m+1$ and let $l_1=\ldots =l_d=l$.       Fix $0<\eta<\nu$ according to Lemma \ref{eta} with all of these choices of parameters.   Fix $n_1\in\nn$ such that for all $n\geqslant n_1$, $b_{p, 2^n}(\mathcal{F}) \leqslant (1+\eta)^{1/p}\Theta$.   Fix $n_2>n_1+nd$ such that $b_{p, 2^{n_2}}(\mathcal{F})>(1-\eta/2)^{1/p}\Theta$.  Let $l_{d+1}=2^{n_2-nd}>2^{n_1}$ and let $L=\prod_{i=1}^{d+1}l_i=2^{n_2}$. Let $T$ be the $(l_1, \ldots, l_{d+1})$ interval tree.   Fix $F:X\to Y\in \mathcal{F}$ and $f:2^L\to X$ such that $$\mathbb{E}\varrho^f_Y(\ee, -\ee)^p>(1-\eta/2)\Theta^p\Bigl(\prod_{i=1}^{d+1}l_i^{p-1}\Bigr)\sum_{i=1}^L \mathbb{E}\varrho_X^f(\ee, d_i\ee)^p.$$    For an interval $I\in T$ and $\ee\in 2^L$, let $I\ee\in 2^L$ be the member of $2^L$ given by  

\begin{displaymath}
   I\ee(i) = \left\{
     \begin{array}{lr}
       \ee(i) & : i\in [L]\setminus I\\
       -\ee(i) & : i\in I.
     \end{array}
   \right.
\end{displaymath}

For each $I\in T$, let $$r_I=\Bigl[\mathbb{E}\varrho^f_Y(\ee, I\ee)^p\Bigr]^{1/p}$$ and $$s_I=\Bigl[\mathbb{E}\varrho^f_X(\ee, I\ee)^p\Bigr]^{1/p}.$$ 

\begin{claim} \begin{enumerate}[(i)]\item $r_{[L]}^p>(1-\eta/2)\Bigl(\prod_{i=1}^{d+1}l_i^{p-1}\Bigr)\Theta^p\sum_{I\in \Lambda_{d+1}} s_I^p$. \item For any $I\in T$, $r_I\leqslant \lambda s_I$.  \item For any $0\leqslant j\leqslant d$ and $I\in \Lambda_j$, $r_I\leqslant \sum_{J^-=I}r_J$ and $s_I\leqslant \sum_{J^-=I}s_j$. \item For all $I\in T$, $r_I^p\leqslant (1+\nu)\Theta^p l_{j+1}^{p-1}\sum_{J^-=I}s_J^p$. \item For any $I\in \Lambda_d$, $r_I^p\leqslant (1+\eta)l_{d+1}^{p-1}\Theta^p \sum_{J^-=I}s^p_J$. \end{enumerate}

\label{claim1}
\end{claim}

\begin{proof}$(i)$ This follows from our choice of $F$, $f$, and the fact that $$r_{[L]}^p=\mathbb{E}\varrho_Y^f(\ee, -\ee)^p$$ and $$\Bigl(\prod_{i=1}^{d+1}l_i^{p-1}\Bigr)\sum_{J\in \Lambda_{d+1}}s^p_I=\Bigl(\prod_{i=1}^{d+1}l_i^{p-1}\Bigr)\sum_{i=1}^L \varrho_X^f(\ee, d_i\ee)^p.$$  

$(ii)$ This follows from the fact that $\text{Lip}(F)\leqslant \lambda$.

$(iii)$ Fix $0\leqslant j\leqslant d$ and $I\in \Lambda_j$. Enumerate $\{J\in T: J^-=I\}$ as $(I_i)_{i=1}^{l_{j+1}}$.    For $i=0, \ldots, l_{j+1}$, define  $J_i:2^L\to 2^L$ by letting $J_0$ be the identity and $J_i=I_iJ_{i-1}$.  For any $\ee\in 2^L$, $J_0\ee=\ee$ and $J_{l_{j+1}}\ee=I\ee$.   Then for any $\ee\in 2^L$, $$\varrho_Y^f(\ee, I\ee)\leqslant \sum_{i=1}^{l_{j+1}}\varrho_Y^f(J_{i-1}\ee, J_i\ee).$$ Now the triangle inequality on $L_p(2^L)$ yields that $$r_I \leqslant \sum_{i=1}^{l_{j+1}} \Bigl[\mathbb{E}\varrho_Y^f(J_{i-1}\ee, J_i\ee)^p\Bigr]^{1/p}.$$   But $\varrho^f_Y(J_{i-1}\ee, J_i\ee)$ and $\varrho_Y^f(\ee, I_i\ee)$ have the same distribution, so $$r_I \leqslant \sum_{i=1}^{l_{j+1}} \Bigl[\mathbb{E}\varrho_Y^f(J_{i-1}\ee, J_i\ee)^p\Bigr]^{1/p} = \sum_{i=1}^{l_{j+1}} \Bigl[\mathbb{E}\varrho_Y^f(\ee, I_i\ee)^p\Bigr]^{1/p} = \sum_{J^-=I}r_J.$$  Replacing each $Y$ with $X$ yields that $s_I\leqslant \sum_{J^-=I}s_J$.

$(iv)$ and $(v)$    Let $I$ and $(I_i)_{i=1}^{l_{j+1}}$ be as in the proof of  $(iii)$.   Define $g:2^L\times 2^{l_{j+1}}\to 2^L$ by letting  \begin{displaymath}
   g(\ee, \delta)(i) = \left\{
     \begin{array}{lr}
       \ee(i) & : i\in [L]\setminus I\\
       \delta(k)\ee(i) & : i\in I_k.
     \end{array}
   \right.
\end{displaymath}  For fixed $\ee\in 2^L$, let $f_\ee:2^{l_{j+1}}\to X$ be given by $f_\ee(\delta)=f(g(\ee, \delta))$.  Note that $g(\ee, -\delta)=Ig(\ee, \delta)$ and $g(\ee, d_i\delta)=I_ig(\ee, \delta)$  for all $\ee\in 2^L$,  $\delta\in 2^{l_{j+1}}$, and $1\leqslant i\leqslant l_{j+1}$. From this it follows that  $f_\ee(-\delta)=f(Ig(\ee, \delta))$ and for $1\leqslant i\leqslant l_{j+1}$, $f_\ee(d_i\ee)=f(I_ig(\ee, \delta))$.    Then \begin{align*} r_I^p & = \mathbb{E}\varrho_Y^f(\ee, I\ee)^p = \mathbb{E}_\delta \mathbb{E}_\ee \varrho_Y^f(g(\ee, \delta), Ig(\ee, \delta))^p \\ & = \mathbb{E}_\ee \mathbb{E}_\delta \varrho_Y^{f_\ee}(\delta, -\delta)^p  \leqslant \mathbb{E}_\ee b_{p, l_{j+1}}(\mathcal{F})^p l_{j+1}^{p-1} \sum_{i=1}^{l_{j+1}}\mathbb{E}_\delta \varrho_X^{f_\ee}(\delta, d_i\delta)^p \\ & = \mathbb{E}_\ee b_{p, l_{j+1}}(\mathcal{F})^p l_{j+1}^{p-1} \sum_{i=1}^{l_{j+1}}\mathbb{E}_\delta \varrho_X^f(g(\ee, \delta), I_ig(\ee, \delta))^p \\ & = b_{p, l_{j+1}}(\mathcal{F})^p l_{j+1}^{p-1} \sum_{i=1}^{l_{j+1}}\mathbb{E}_\delta \mathbb{E}_\ee\varrho_X^f(g(\ee, \delta), I_ig(\ee, \delta))^p \\ & = b_{p, l_{j+1}}(\mathcal{F})^p l_{j+1}^{p-1} \sum_{i=1}^{l_{j+1}} \mathbb{E}_\ee\varrho_X^f(\ee, I_i\ee)^p= b_{p, l_{j+1}}(\mathcal{F})^p l_{j+1}^{p-1}\sum_{J^-=I} s_J^p.\end{align*} The fact that $b_{p, l_{j+1}}(\mathcal{F})\leqslant (1+\nu)^{1/p}\Theta$ for all $j$ yields $(iv)$, and the fact that $b_{p, l_{d+1}}(\mathcal{F}) \leqslant (1+\eta)^{1/p}\Theta$ yields $(v)$.

\end{proof}

\begin{claim} There exist $0\leqslant j_0<d$ and $I\in \Lambda_{j_0}$ such that \begin{enumerate}[(i)]\item $r_I^p>(1-\mu)l^{p-1}\sum_{J^-=I}r_J^p$, \item $s_I^p>(1-\mu)l^{p-1}\sum_{J^-=I}s_J^p$, and \item $l^{p-1}\sum_{J^-=I}r^p_J\geqslant r_I>(1-\nu)\Theta^p l^{p-1} \sum_{J^-=I}s^p_J$. \end{enumerate}

\label{claim2}
\end{claim}

\begin{proof} For each $0\leqslant j<d$, let $$A_j=\Bigl\{I\in \Lambda_j: s_I^p\leqslant (1-\mu)l^{p-1}\sum_{J^-=I} s_J^p\Bigr\}$$ and $$B_j=\Bigl\{I\in \Lambda_j: r_I^p\leqslant (1-\nu)\Theta^pl^{p-1}\sum_{J^-=I}s_J^p\Bigr\}. $$     By Lemma \ref{goodX}, there are at most $m$ values of $j<d$ such that $|A_j|\geqslant \Delta|\Lambda_j|$. Here we note that $$s_{[L]}^p\geqslant \frac{r_{[L]}^p}{\lambda^p}> (1-\eta/2)\frac{\Theta^p}{\lambda^p}\Bigl(\prod_{i=1}^{d+1}l_i^{p-1}\Bigr)\sum_{I\in \Lambda_{d+1}}s_I^p.$$    Since $m+1<d$, there exists at least one value $j_0<d$ such that $|A_{j_0}|<\Delta |\Lambda_{j_0}|$. By Lemma \ref{YtoX}, for this $j_0$, $|B_{j_0}|< (1-\Delta)|\Lambda_{j_0}|$.  Thus $$|A_{j_0}|+|B_{j_0}|<|\Lambda_{j_0}|,$$ whence there exists $I\in \Lambda_{j_0}\setminus (A_{j_0}\cup B_{j_0})$. Since $I\in \Lambda_{j_0}\setminus A_{j_0}$, $(ii)$ is satisfied for this $I$. Since $I\in \Lambda_{j_0}\setminus B_{j_0}$ and since $l^{p-1}\sum_{J^-=I}r_J^p \geqslant r_I^p$,  $(iii)$ is satisfied for this $I$. Since $j_0<d$ and $I\in \Lambda_{j_0}$, $(i)$ is satisfied for this $I$ by Lemma \ref{eta}.

\end{proof}

For the remainder of the proof, $I$ is the fixed interval from Claim \ref{claim2}.   Enumerate $\{J\in T: J^-=I\}$ as $(I_i)_{i=1}^l$ and define $g:2^L\times 2^l\to 2^L$ by letting \begin{displaymath}
   g(\ee, \delta)(i) = \left\{
     \begin{array}{lr}
       \ee(i) & : i\in [L]\setminus I\\
       \delta(j)\ee(i) & : i\in I_j.
     \end{array}
   \right.
\end{displaymath}  We also define the functions $J_0, \ldots, J_l:2^L\to 2^L$ by letting $J_0$ be the identity function and $J_i=I_iJ_{i-1}$. Note that $\ee=J_0\ee$ and $J_l\ee=I\ee$.

Now define $D_Y, E_Y, D_X, E_X:2^L\to [0, \infty)$ by $$D_Y(\ee)= \mathbb{E}_\delta \varrho^f_Y(g(\ee, \delta), Ig(\ee, \delta))^p,$$ $$E_Y(\ee)=l^{p-1}\sum_{i=1}^l \mathbb{E}_\delta \varrho_Y^f(J_{i-1}g(\ee, \delta), J_i g(\ee, \delta))^p,$$  $$D_X(\ee)= \mathbb{E}_\delta \varrho^f_X(g(\ee, \delta), Ig(\ee, \delta))^p,$$ and $$E_X(\ee)=l^{p-1}\sum_{i=1}^l \mathbb{E}_\delta \varrho_X^f(J_{i-1}g(\ee, \delta), J_i g(\ee, \delta))^p.$$

\begin{claim} The functions $D_Y, E_Y, D_X, E_X$ satisfy the hypotheses of Lemma \ref{expect}. 

\label{claim3}
\end{claim}

\begin{proof} $D_Y\leqslant \lambda^p D_X$ and $E_Y\leqslant \lambda^p E_X$ follow from the fact that $\text{Lip}(F)\leqslant \lambda$.   For a fixed $\delta\in 2^l$ and $\ee\in 2^L$, $$\varrho_Y^f(g(\ee, \delta), Ig(\ee, \delta))^p\leqslant l^{p-1}\sum_{i=1}^l \varrho_Y^f(J_{i-1}g(\ee, \delta), J_ig(\ee, \delta))^p$$ follows from the triangle and H\"{o}lder inequalities. Taking expectations over $\delta$ with $\ee$ held fixed yields that $D_Y\leqslant E_Y$. Replacing $Y$ with $X$ yields that $D_X\leqslant E_X$.   Now fix $\ee\in 2^L$ and define $f_\ee(\delta)=g(\ee, \delta)$.      Define $d_{<i}, d_{\leqslant i}:2^l\to 2^l$ by letting $d_{<i}\delta$ be the member of $2^l$ by replacing $\delta(k)$ with $-\delta(k)$ for each $k<i$ and leaving the other coordinates of $\delta$ unchanged. Let $d_{\leqslant i}=d_id_{<i}$.    Note that $\varrho_X^{f_\ee}(\delta, d_i\delta)^p$ and $\varrho_X^{f_\ee}(d_{<i}\delta, d_{\leqslant i\delta})$ have the same distribution as functions of $\delta\in 2^l$.  Note also that $g(\ee, d_{<i}\delta)=J_{i-1}g(\ee, \delta)$ and $g(\ee, \delta_{\leqslant i}\delta)=J_i g(\ee, \delta)$.    Then \begin{align*} D_Y(\ee) & = \mathbb{E}_\delta\varrho_Y^f(g(\ee, \delta), Ig(\ee, \delta))^p = \mathbb{E}_\delta \varrho_Y^{f_\ee}(\delta, -\delta)^p \\ & \leqslant b_{p,l}(\mathcal{F})^p l^{p-1} \sum_{i=1}^l \mathbb{E}_\delta \varrho_X^{f_\ee}(\delta, d_i\ee)^p = b_{p,l}(\mathcal{F})^pl^{p-1}\sum_{i=1}^l \mathbb{E}_\delta \varrho_X^{f_\ee}(d_{<i}\delta, d_{\leqslant i}\delta)^p \\ & \leqslant (1+\nu)\Theta^p l^{p-1} \sum_{i=1}^l \mathbb{E}_\delta \varrho_X^f(J_{i-1} g(\ee, \delta), J_i g(\ee, \delta))^p = (1+\nu)\Theta^p E_X(\ee). \end{align*} This yields that $D_Y, E_Y, D_X, E_X$ satisfy hypothesis $(i)$ of Lemma \ref{expect}.

For a fixed $\delta\in 2^l$ and $1\leqslant i\leqslant l$,  $\varrho_Y^f(\ee, I\ee)$ and $\varrho_Y^f(g(\ee, \delta), Ig(\ee, \delta))$ have the same distribution as functions of $\ee\in 2^L$, as do $\varrho_Y^f(\ee, I_i\ee)$ and  $\varrho_Y^f(J_{i-1}g(\ee, \delta), J_ig(\ee, \delta))$. The analogous statements hold with each  $Y$ replaced by $X$. By exchanging order of integration of $\ee$ and $\delta$, we see that $$\mathbb{E}_\ee D_Y(\ee)=\mathbb{E}_\ee\mathbb{E}_\delta  \varrho_Y^f(g(\ee, \delta), Ig(\ee, \delta))^p = \mathbb{E}_\delta \mathbb{E}_\ee \varrho^f_Y (g(\ee, \delta), Ig(\ee, \delta))^p = \mathbb{E}_\ee \varrho^f_Y (\ee, I\ee)^p=r_I^p.$$   We similarly deduce that  $\mathbb{E}_\ee E_Y(\ee)=l^{p-1}\sum_{J^-=I} r_J^p$, $\mathbb{E}_\ee D_X(\ee)=s_I^p$, and $\mathbb{E}_\ee E_X(\ee)=l^{p-1}\sum_{J^-=I} s_J^p$. Thus hypothesis $(ii)$ of Lemma \ref{expect} is satisfied because $I$ satisfies the conclusions of Claim \ref{claim2}.

Hypothesis $(iii)$ of Lemma \ref{expect} is satisfied by our chocies of $\mu, a, \nu$, and $b$.

\end{proof}

Now by Lemma \ref{expect}, there exists $\ee_0\in 2^L$ such that $$D_Y(\ee_0)>(1-a)E_Y(\ee_0),$$ $$D_X(\ee_0)>(1-a)E_X(\ee_0),$$ and $$E_Y(\ee_0)\geqslant D_Y(\ee_0)>(1-b)\Theta^p E_X(\ee_0).$$    Now define $h:2^l\to X$ by letting $h(\delta)=f(g(\ee_0, \delta))$.   Let $$r^p=\frac{1}{l}\sum_{i=1}^l \mathbb{E}_\delta \varrho_Y^h( \delta, d_i\delta)^p= E_Y(\ee_0)/l^p$$ and $$s^p=\frac{1}{l}\sum_{i=1}^l \mathbb{E}_\delta \varrho_X^h(\delta,d_i\delta)^p=E_X(\ee_0)/l^p \leqslant (1-b)\Theta^p r^p.$$    Then since $$\mathbb{E}_\delta \varrho^h_Y(\delta, -\delta)^p=D_Y(\ee_0)>(1-a)E_Y(\ee_0)=(1-a)l^{p-1}\sum_{i=1}^l \mathbb{E}_\delta \varrho_Y^h(\delta, d_i\delta)^p$$ and $$\mathbb{E}_\delta \varrho^h_X(\delta, -\delta)^p=D_X(\ee_0)>(1-a)E_X(\ee_0)=(1-a)l^{p-1}\sum_{i=1}^l \mathbb{E}_\delta \varrho_X^h(\delta, d_i\delta)^p,$$ it follows from our choice of $a$ and Theorem \ref{BMW} that for any $\delta_1, \delta_2\in 2^l$, $$\frac{1}{D}\partial (\delta_1, \delta_2) \leqslant \frac{d_Y(h(\delta_1), h(\delta_2))}{r} \leqslant D\partial (\delta_1, \delta_2)$$ and $$\frac{1}{D}\partial (\delta_1, \delta_2) \leqslant \frac{d_X(F\circ h(\delta_1), F\circ h(\delta_2))}{s} \leqslant D\partial (\delta_1, \delta_2).$$   Since $r\geqslant \vartheta s$, this finishes the proof.

\begin{rem}\upshape We observe the following quantitative consequence of the previous proof and our remark from the introduction. If we define $c(\mathcal{F})$ to be the supremum of those $c>0$ such that for each $D>1$ and $n\in \nn$, there exist $F:X\to Y\in \mathcal{F}$, $f:2^n\to X$, and $a,b>0$ such that $b\geqslant ac$ and for each $\ee_1, \ee_2\in 2^n$, $$\frac{a}{D} \partial (\ee_1, \ee_2) \leqslant d_X(f(\ee_1), f(\ee_2)) \leqslant aD \partial (\ee_1, \ee_2)$$ and $$\frac{b}{D}\partial (\ee_1, \ee_2) \leqslant d_Y(F\circ f(\ee_1), F\circ f(\ee_2)) \leqslant bD \partial (\ee_1, \ee_2),$$ then $$\underset{n}{\lim\sup} b_{p,n}(\mathcal{F}) \geqslant c(\mathcal{F}) \geqslant \underset{n}{\lim\sup} b_{p,2^n}(\mathcal{F}) \geqslant 2^{1/p-1}\underset{n}{\lim\sup} b_{p,n}(\mathcal{F}).$$   

\label{gg}

\end{rem}

\section{The quantities $a_n(\mathcal{F})$}

The goal of this section is to prove the implication $(iii)\Rightarrow (i)$ with the additional quantitative information: If $(i)$ is satisfied and $c(\mathcal{F})$ is as defined in Remark \ref{gg}, then $c(\mathcal{F})=\lim\sup_n a_n(\mathcal{F})=\lim_n a_n(\mathcal{F})$.     It is obvious that $c(\mathcal{F})\leqslant \lim\sup_n a_n(\mathcal{F})$, so we establish the following criterion for obtaining the reverse inequality.  The basis of this criterion is to use standard self-improvement arguments for embeddings into $X$ without worsening the scaling factors between the embeddings of the cube into $X$ via some $f$ and the corresponding embedding of the cube into $Y$ via $F\circ f$.

\begin{lemma} Suppose $\lambda=\sup_{F\in \mathcal{F}} \text{\emph{Lip}}(F)\in (0,\infty)$.  If $\Theta>0$ is such that for any $0<\vartheta<\Theta$, $D>1$, and $l\in\nn$, there exist $F:X\to Y\in \mathcal{F}$ and $h:2^l\to X$ such that $\text{\emph{dist}}(F\circ h)\leqslant D$ and $\varrho_Y^h(\ee_1, \ee_2)\geqslant \vartheta d_X^h(\ee_1, \ee_2)$ for all $\ee_1, \ee_2\in 2^l$, then $c(\mathcal{F})\geqslant \Theta$. 

\label{taco bell}

\end{lemma}

\begin{proof} First fix $0<\vartheta<\Theta$. For each $l\in\nn$, let  $\xi_n=\xi_n(\mathcal{F})$ be the supremum of those constants $\xi>0$ such that for all $D>1$, there exist $F:X\to Y\in \mathcal{F}$ and $f:2^n\to X$ such that $\text{dist}(F\circ f)\leqslant D$, for each $\ee_1, \ee_2\in 2^n$, $\varrho_Y^f(\ee_1, \ee_2) \geqslant \vartheta d_X^f(\ee_1, \ee_2)$, and $$\mathbb{E}\varrho_X^f(\ee, -\ee)^2>n \xi^2 \sum_{i=1}^n \mathbb{E} \varrho_X^f(\ee, d_i\ee)^2.$$

Let us observe the following facts: \begin{enumerate}[(i)]\item For all $n\in\nn$, $\vartheta/\lambda\leqslant \xi_n\leqslant 1$. \item For all $m,n\in\nn$, $\xi_{mn}\leqslant \xi_m\xi_n$. \end{enumerate}

For the first fact, $\xi_n\leqslant 1$ follows as usual from the triangle and H\"{o}lder inequalities.       By hypothesis, for each $D_1,D_2>1$, there exist $F:X\to Y\in \mathcal{F}$ and $f:2^n\to X$ such that $\text{dist}(F\circ f)\leqslant D=\min\{D_1, D_2\}$ and $\vartheta \varrho_X^f(\ee_1, \ee_2) \leqslant \varrho_Y^f(\ee_1, \ee_2)$ for all $\ee_1, \ee_2\in 2^n$.    Then \begin{align*} \lambda^2 \mathbb{E}\varrho_X^f(\ee,-\ee)^2 & \geqslant \mathbb{E}\varrho_Y^f(\ee_1, -\ee)^2 \geqslant D^{-4}n\sum_{i=1}^n \varrho_Y^f(\ee, d_i\ee)^2 \geqslant \vartheta^2 D^{-4}n\sum_{i=1}^n \varrho_X^f(\ee, d_i\ee)^2. \end{align*}  Now applying this as $D_1\downarrow 1$ with $D_2>1$ held fixed, we deduce that $\xi_n\geqslant \vartheta/\lambda D^2_2$.   Unfixing $D_2>1$ yields that $\xi_n\geqslant \vartheta/\lambda$.

For the second item, suppose $\xi_{mn}>\xi_m\xi_n$ for some $m,n\in\nn$.  Fix $\alpha>\xi_m$ and $\beta>\xi_n$ such that $\alpha\beta<\xi_{mn}$.     By definition of $\xi_m$, there exists $D_1>1$ such that for all $F:X\to Y\in \mathcal{F}$ and $f:2^m\to X$ with $\text{dist}(F\circ f)\leqslant D_1$ and $\vartheta \varrho_X^f(\ee_1, \ee_2)\leqslant \varrho_Y^f(\ee_1, \ee_2)$ for all $\ee_1, \ee_2\in 2^m$, $$\mathbb{E}\varrho_X^f(\ee, -\ee)^2 \leqslant m\alpha^2 \sum_{i=1}^m \mathbb{E}\varrho_X^f(\ee, d_i\ee)^2.$$    Similarly, there exists $D_2>1$ such that for any $F:X\to Y\in \mathcal{F}$ and $f:2^n\to X$ with $\text{dist}(F\circ f)\leqslant D_2$ and $\vartheta \varrho_X^f(\ee_1, \ee_2)\leqslant \varrho_Y^f(\ee_1, \ee_2)$ for all $\ee_1, \ee_2\in 2^n$, $$\mathbb{E}\varrho_X^f(\ee, -\ee)^2 \leqslant n\beta^2\sum_{i=1}^n \mathbb{E}\varrho_X^f(\ee, d_i\ee)^2.$$   Let $D=\min\{D_1, D_2\}>1$ and fix $\xi, F, f$ such that  $\alpha\beta<\xi<\xi_{mn}$ and  $F:X\to Y\in \mathcal{F}$,  $f:2^{mn}\to X$ satisfy $\text{dist}(F\circ f)\leqslant D$, $\vartheta \varrho_X^f(\ee_1, \ee_2)\leqslant \varrho_Y^f(\ee_1, \ee_2)$ for all $\ee_1, \ee_2\in 2^{mn}$, and $$\mathbb{E}\varrho_X^f(\ee, -\ee)^2>mn\xi^2\sum_{i=1}^{mn} \varrho_X^f(\ee, d_i\ee)^2.$$   Now as usual, let $I_1, \ldots, I_m$ be a partition of $[mn]$ into intervals of cardinality $n$ and define \begin{displaymath}
   I_j\ee(i) = \left\{
     \begin{array}{lr}
       \ee(i) & : i\in [mn]\setminus I_j\\
       -\ee(i) & : i\in I_j
     \end{array}
   \right.
\end{displaymath} 

 Define $g:2^{mn}\times 2^m\to 2^{mn}$ be defined by $g(\ee, \delta)(i)=\delta(j)\ee(i)$, where $i\in I_j$.   We identify $\ee$ with $(\ee_i)_{i=1}^m$, where $(\ee((i-1)n+1), \ldots, \ee(in))=\ee_i\in 2^n$. For $\ee\in 2^{mn}$, we let $\ee_{-i}\in 2^{(m-1)n}$ be defined by $\ee_{-i}=(\ee_1, \ldots, \ee_{i-1}, \ee_{i+1}, \ldots, \ee_m)$.     Note that for a fixed $\ee_{-i}$ and a fixed $\ee'\in 2^{mn}$, if $h$ is the map from $2^n$ to $X$ given by $\ee\mapsto f((\ee_1, \ldots, \ee_{i-1}, \ee, \ee_{i+1}, \ldots, \ee_m))$ or if $h$ is the map from $2^m$ to $X$ given by $\delta\mapsto f(g(\ee', \delta))$, then $\text{dist}(F\circ h)\leqslant \text{dist}(F\circ f)\leqslant D$ and $\vartheta \varrho_X^h(\ee_1, \ee_2)\leqslant \varrho_Y^h(\ee_1, \ee_2)$ for all $\ee_1, \ee_2\in 2^n$ (resp. $\ee_1, \ee_2\in 2^m$).     Then \begin{align*} mn\xi^2\sum_{i=1}^{mn} \mathbb{E} \varrho_X^f(\ee, d_i\ee)^2 & < \mathbb{E}\varrho_X^f(\ee, -\ee)^2  =\mathbb{E}_\ee \mathbb{E}_\delta \varrho_X^f(g(\ee, \delta), g(\ee, -\delta))^2 \\ & \leqslant m\alpha^2 \sum_{i=1}^m \mathbb{E}_\ee\mathbb{E}_\delta \varrho_X^f(g(\ee, \delta), g(\ee, d_i\delta))^2  = m\alpha^2 \sum_{i=1}^m \mathbb{E}_\ee\mathbb{E}_\delta \varrho_X^f(g(\ee, \delta), I_i g(\ee, \delta))^2 \\ &  = m\alpha^2 \sum_{i=1}^m \mathbb{E} \varrho_X^f(\ee, I_i\ee)^2 = m \alpha^2\sum_{i=1}^m \mathbb{E}_{\ee_{-i}}\mathbb{E}_{\ee_i} \varrho_X^f(\ee, I_i\ee)^2 \\ & \leqslant mn\alpha^2\beta^2 \sum_{i=1}^m \mathbb{E}_{\ee_{-i}}\sum_{j=(i-1)+1}^{ij} \mathbb{E}_{\ee_i}\varrho_X^f(\ee, d_j\ee)^2  =mn\alpha^2\beta^2\sum_{i=1}^{mn}\mathbb{E}\varrho_X^f(\ee, d_i\ee)^2. \end{align*}  This contradiction yields $(ii)$.

Now since $(\xi_n)_{n=1}^\infty$ is submultiplicative and lies in $[\vartheta/\lambda, 1]$, it must be that $\xi_n=1$ for all $n\in\nn$. Indeed, if $\xi_n<1$, then for large enough $t\in \nn$, $\xi_{n^t}\leqslant \xi_n^t<\vartheta/\lambda$.    Now fix $n\in\nn$ and $D>1$.  By Theorem \ref{BMW}, there exists $0<\mu<1$ such that if $f:2^n\to X$ satisfies $$\mathbb{E}\varrho_X^f(\ee, -\ee)^2>(1-\mu)n\sum_{i=1}^n \mathbb{E}_X^f(\ee, d_i\ee)^2,$$ then $\text{dist}(f)\leqslant D$.    By the definition of $\xi_n$ and since $\xi_n=1>1-\mu$,  there exist $F:X\to Y\in \mathcal{F}$ and $f:2^n\to X$ such that $\text{dist}(F\circ f)\leqslant D$, $\vartheta \varrho_X^f(\ee_1, \ee_2)\leqslant \varrho_Y^f(\ee_1, \ee_2)$ for all $\ee_1, \ee_2\in 2^n$, and $$\mathbb{E}\varrho_X^f(\ee, -\ee)^2>(1-\mu)n\sum_{i=1}^n \mathbb{E}\varrho_X^f(\ee, d_i\ee)^2.$$  From this it follows that $\text{dist}(f)\leqslant D$.   Moreover, if $a,b>0$ are such that $$\frac{a}{D}\partial (\ee_1, \ee_2)\leqslant \varrho_X^f(\ee_1, \ee_2) \leqslant aD\partial (\ee_1, \ee_2)$$ and $$\frac{b}{D}\partial (\ee_1, \ee_2)\leqslant \varrho_Y^f(\ee_1, \ee_2) \leqslant bD\partial (\ee_1, \ee_2)$$ for all $\ee_1, \ee_2$, then $a\vartheta/D\leqslant bD$.  Since $D>1$, $n\in\nn$ are arbitrary, $c(\mathcal{F})\geqslant \vartheta$. Now we unfix $0<\vartheta<\Theta$ and deduce that $c(\mathcal{F})\geqslant \Theta$.

\end{proof}

\begin{proposition} If $\mathcal{F}$ is a uniformly Lipschitz collection of maps, then $\lim_n a_n(\mathcal{F})=\inf_n a_n(\mathcal{F})$. 

\end{proposition}

\begin{proof} Let $\lambda=\sup_{F\in \mathcal{F}}\text{Lip}(F)\in (0, \infty)$.  Fix $k,l\in \nn$. Let $I_1, \ldots, I_l$ be a partition of $[kl]$ into subintervals of cardinality $k$. Let $I_j:2^{kl}\to 2^{kl}$ be such that  \begin{displaymath}
   I_j\ee(i) = \left\{
     \begin{array}{lr}
       \ee(i) & : i\in [kl]\setminus I_j\\
       -\ee(i) & : i\in I_j.
     \end{array}
   \right.
\end{displaymath}  

Define $g:2^{kl}\times 2^l\to 2^{kl}$ by $g(\ee, \delta)(i)=\delta(j)\ee(i)$, where $j$ is such that $i\in I_j$.

  Now fix any $F:X\to Y\in \mathcal{F}$ and $f:2^{kl}\to X$.  Note that for fixed $\ee\in 2^{kl}$, the function $\delta\mapsto g(\ee, \delta)$ is distance preserving.  Therefore for each fixed $\ee\in 2^l$, the function $f_\ee:2^l\to X$ given by $f_\ee(\delta)=f(g(\ee, \delta))$ satisfies $\text{Lip}(f_\ee)\leqslant  \text{Lip}(f)$.    Then \begin{align*} \mathbb{E}_\ee \varrho_Y^f(\ee, -\ee) & = \mathbb{E}_\delta \mathbb{E}_\ee \varrho_Y^f(g(\ee, \delta), g(\ee, -\delta)) = \mathbb{E}_\ee \mathbb{E}_\delta \varrho^{f_\ee}_Y(\delta, -\delta) \\ & \leqslant a_l(\mathcal{F})\mathbb{E}_\ee \text{Lip}(f_\ee)\leqslant a_l(\mathcal{F}) \text{Lip}(f). \end{align*}  From this it follows that $a_{kl}(\mathcal{F})\leqslant a_l(\mathcal{F})$.

Now fix $l\in\nn$.  For $m>l$, write $m=k_ml+r_m$ where $k_m\in\nn$ and $0\leqslant r_m<k$. For the moment, we suppress the subscript $m$ and simply write $k_m=k$ and $r_m=r$.    Now fix $F:X\to Y\in \mathcal{F}$ and $f:2^m\to X$.    Let $\delta=(\delta(1), \ldots, \delta(r))\in 2^r$ be arbitrary and define $g:2^{lk}\to 2^m$ by $g(\ee)=(\ee(1), \ldots, \ee(lk), \delta(1), \ldots, \delta(r))$. Define $h:2^m\to 2^{lk}$ by $h(\ee)=(\ee(1), \ldots, \ee(lk))$.         Note that for each $\ee_1, \ee_2\in 2^{lk}$, $$\frac{1}{lk}\partial(\ee_1, \ee_2)= \frac{1}{m}\partial (g(\ee_1), g(\ee_2)).$$  Therefore the map $G:2^{lk}\to X$ given by $G(\ee)=f(g(\ee))$ has $\text{Lip}(G)\leqslant \frac{m}{lk}\partial (\ee_1, \ee_2)$. Define $H:2^m\to X$ by $H(\ee)=G(h(\ee))=f(g(h(\ee)))$.  Let us also note that $\mathbb{E}\varrho_Y^G(\ee, -\ee)=\mathbb{E}\varrho_Y^H(\ee, -\ee)$.  For any $\ee\in 2^m$, since $\ee$ and $g(h(\ee))$ differ in at most $r$ coordinates, $$\varrho_Y^f(\ee, g(h(\ee))) \leqslant \lambda \text{Lip}(f)r/m.$$    Therefore \begin{align*} \mathbb{E}\varrho_Y^f(\ee, -\ee) & \leqslant \mathbb{E}\varrho_Y^f(g(h(\ee)), g(h(-\ee)) + \mathbb{E}\varrho_Y^f(\ee, g(h(\ee))+ \mathbb{E}\varrho_Y^f(g(h(-\ee)), -\ee) \\ &  \leqslant  \mathbb{E}\varrho_Y^f(g(h(\ee)), g(h(-\ee))+ 2\lambda \text{Lip}(f)r/m  = \mathbb{E}\varrho_Y^H (\ee, -\ee)+2\lambda \text{Lip}(f)r/m \\ & = \mathbb{E}\varrho_Y^G(\ee, -\ee)+2\lambda \text{Lip}(f)r/m  \leqslant a_{lk}(\mathcal{F})\text{Lip}(G)+2\lambda \text{Lip}(f)r/m \\ & \leqslant \Bigl[a_l(\mathcal{F})\frac{m}{lk}+\frac{2\lambda r}{m}\Bigr]\text{Lip}(f). \end{align*}    From this it follows that $$a_m(\mathcal{F}) \leqslant a_l(\mathcal{F})\frac{m}{lk}+\frac{2\lambda r}{m}.$$  Now once more writing $k=k_m$ and $r=r_m$ and noting that $\frac{m}{lk_m}\to 1$ and $\frac{2\lambda r_m}{m}\to 0$ as $m\to \infty$, we deduce that $a_l(\mathcal{F})\geqslant \lim\sup_m a_m(\mathcal{F})$. Since $l\in \nn$ was arbitrary, we are done.

\end{proof}

\begin{rem}\upshape If $L,l\in\nn$ are two natural numbers, $g:2^L\times 2^l\to 2^L$ is a function such that for each $\delta\in 2^l$, $\ee\mapsto g(\ee, \delta)$ is a bijection, and $\Omega\subset 2^L$ is such that $\mathbb{P}(\Omega)<1/2^l$, then there exists $\ee\in 2^L$ such that $\{g(\ee, \delta):\delta\in 2^l\}\subset \Omega^c$. Indeed,  for each $\ee_0\in 2^L$ and $\delta_0\in 2^l$,  define $\Omega_{\ee_0}=\{\delta\in 2^l: g(\ee_0, \delta)\in \Omega\}$,  $\Omega^{\delta_0}=\{\ee\in 2^L: g(\ee, \delta_0)\in \Omega\}$, and $\Omega_1=\{(\ee, \delta)\in 2^L\times 2^l: g(\ee, \delta)\in \Omega\}$. Then if for each $\ee$, there exists $\delta\in 2^l$ such that $g(\ee, \delta)\in \Omega^c$,  \begin{align*} \mathbb{P}(\Omega)=\frac{1}{2^l}\sum_{\delta\in 2^l}\mathbb{P}(\Omega)=\frac{1}{2^l}\sum_{\delta\in 2^l} \mathbb{P}(\Omega^\delta) = \mathbb{P}(\Omega_1)  = \frac{1}{2^L}\sum_{\ee\in 2^L}\mathbb{P}(\Omega_\ee) \geqslant \frac{1}{2^L}\sum_{\ee\in 2^L}1/2^l = 1/2^l.\end{align*}

\label{density}
\end{rem}

Let us also recall the following simple consequence of the reverse triangle inequality, which we use as a substitute for Theorem \ref{BMW} in this section. 

\begin{proposition} For each $l\in\nn$ and $D>1$, there exists $0<a<1$ such that if $(Z, d_Z)$ is any metric space and $h:2^l\to Z$ is a map  such that $(1-a)\text{\emph{Lip}}(h) < \min_{\delta\in 2^l} d_Z(h(\delta), h(-\delta))$, then $h$ is an embedding with distortion not more than $D$.

\label{sharp}
\end{proposition}

\begin{proof} Fix $0<a<1/l$ so small that $1-al>1/D$.     Fix $\delta\neq \delta_1\in 2^l$ and let $m=l \partial (\delta, \delta_1)$.  Then \begin{align*} d_Z(h(\delta), h(\delta_1)) & \geqslant d_Z(h(\delta), h(-\delta))-d_Z(h(-\delta), h(\delta_1)) \geqslant (1-a)\text{Lip}(h)-\frac{l-m}{l}\text{Lip}(h) \\ & = \Bigl(\frac{m}{l}-a\Bigr)\text{Lip}(h) \geqslant \frac{m}{l}(1-al)\text{Lip}(h) > \frac{\text{Lip}(h)}{D} \partial (\delta, \delta_1).   \end{align*}

From this it follows that $\text{Lip}(h^{-1})\leqslant D/\text{Lip}(h)$, so $\text{Lip}(h)\text{Lip}(h^{-1})\leqslant D$.

\end{proof}

We next recall the concentration of measure for the Hamming cube. 

\begin{lemma}\cite{Harper,AM} There exist constants $\alpha, \beta>0$ such that for any $n\in\nn$ and $\lambda_1>0$, if $\Phi:2^n\to \rr$ is $\lambda_1$-Lipschitz and if $\phi$ is a median of $\Phi$, then for any $t>0$,  $$\mathbb{P}\Bigl(|\Phi-\phi|>t\lambda_1\Bigr)\leqslant \alpha \exp(-\beta t n).$$

\label{concentrate}

\end{lemma}

\begin{proof}[Proof of $(iii)\Rightarrow (i)$]  Let $\lambda=\sup_{F\in\mathcal{F}}\text{Lip}(F)\in (0,\infty)$ and $\Theta=\lim_n a_n(\mathcal{F})\in (0, \lambda]$. Fix $0<\vartheta<\Theta$. Fix $l\in\nn$ and $D>1$ such that $\vartheta<\Theta/D$.  Let $0<a<1$ be chosen according to Proposition \ref{sharp}.  Now fix $0<\mu<1/2$ such that $\frac{1+\mu}{1-\mu}>1-a$ and $\frac{(1-2\mu)}{D}\Theta>\vartheta$.   Fix $0<\eta<\mu$ such that $$\frac{1-\mu}{l}+(1+\eta)\bigl(\frac{l-1}{l}\bigr)<1-\eta.$$    Fix $t>0$ such that $t<\mu\Theta/l$ and $$(l+1)t< \mu(1-2\mu)\Theta.$$  Note that the second inequality implies that for any $M\geqslant (1-2\mu)\Theta$, $$(l+1)t+M\leqslant (1+\mu)M.$$     Fix  $k\in\nn$ so large that for all $m\geqslant k$, $a_m(\mathcal{F})\in ((1-\eta)\Theta, (1+\eta)\Theta)$, $\lambda \alpha \exp(-\beta tm/8\lambda)<t/4\lambda$, and $(l+1)\alpha \exp(-\beta t m/8\lambda)<1/2^l$.   Fix $F:X\to Y\in \mathcal{F}$ and $f:2^{lk}\to X$ such that $$\mathbb{E}\varrho_Y^f(\ee, -\ee)>(1-\eta)\Theta \text{Lip}(f).$$  Let $T$ be the $(l,k)$ interval tree and for each $I\in T$, define $\Phi_I:2^{lk}\to \rr$ by $\Phi_I(\ee)=\varrho_Y^f(\ee, I\ee)$, where $I\ee$ is obtained by changing the signs of the coordinates of $\ee$ which lie in $I$ and leaving the other coordinates unchanged.     Let $\phi_I$ be a median of $\Phi_I$ and let $r_I=\mathbb{E}\Phi_I$. For the remainder of the proof, fix a partition of $[lk]$ into intervals $I_1, \ldots, I_l$, where $I_j=\{(j-1)k+1, \ldots, jk\}$.

We first claim that $$(1-\eta)\Theta \text{Lip}(f) <r_{[lk]}< (1+\eta)\Theta \text{Lip}(f)$$ and for each $I\in \Lambda_1$, $$(1-\mu)\Theta \text{Lip}(f)\leqslant  l r_I \leqslant (1+\eta)\Theta \text{Lip}(f).$$  The first pair of inequalities follows from the fact that $r_{[lk]}=\mathbb{E}\varrho_Y^f(\ee, -\ee)\leqslant a_{lk}(\mathcal{F})\text{Lip}(f)< (1+\eta)\Theta \text{Lip}(f)$ and $F,f$ were chosen such  that $(1-\eta)\Theta \text{Lip}(f)<\mathbb{E}\varrho_Y^f(\ee, -\ee)$.    Now fix $1\leqslant j\leqslant l$ and define $g:2^{lk}\times 2^k\to 2^{lk}$ by letting  \begin{displaymath}
   g(\ee, \delta)(i) = \left\{
     \begin{array}{lr}
       \ee(i) & : i\in [lk]\setminus I_j\\
       \delta(m)\ee(i) & : i=(j-1)k+m.
     \end{array}
   \right.
\end{displaymath}  Note that for a fixed $\ee\in 2^{lk}$, the map $f_\ee:2^l\to 2^{lk}$ given by $f_\ee(\delta)\mapsto g(\ee, \delta)$ scales distances  by a factor of $1/l$. From this it follows that for a fixed $\ee$, the map  $\delta\mapsto f(g(\ee, \delta))$ has Lipschitz constant not more than $\text{Lip}(f)/l$.    Therefore \begin{align*} r_{I_j} & = \mathbb{E}_\ee\varrho_Y^f(\ee, I_j\ee) = \mathbb{E}_\ee \mathbb{E}_\delta \varrho_Y^f(g(\ee, \delta), g(\ee, -\delta)) \leqslant a_k(\mathcal{F})\mathbb{E}_\ee\text{Lip}(f_\ee) \leqslant (1+\eta)\Theta \text{Lip}(f)/l. \end{align*}   From this we deduce that $$\max_{I\in \Lambda_1} lr_I \leqslant (1+\eta)\Theta\text{Lip}(f).$$   To see that $(1-\mu)\Theta \text{Lip}(f) \leqslant lr_I$ for all $I\in \Lambda_1$, suppose that there exists $I_0\in \Lambda_1$ such that $r_{I_0}<(1-\mu)\Theta \text{Lip}(f)/l$.   Then \begin{align*} (1-\eta)\Theta \text{Lip}(f) & < r_{[lk]}  \leqslant \sum_{I\in \Lambda_1} r_I <(1-\mu)\Theta \text{Lip}(f)/l+(l-1)(1+\eta)\Theta \text{Lip}(f)/l \\ & = \Bigl[\frac{1-\mu}{l}+ (1+\eta)\bigl(\frac{l-1}{l}\bigr)\Bigr]\Theta \text{Lip}(f) < (1-\eta)\Theta \text{Lip}(f).\end{align*} This is a contradiction and yields the remaining inequality.  Here we are using the fact that $r_{[lk]}\leqslant \sum_{I\in \Lambda_1} r_I$, which follows from the triangle inequality as in the proof from the previous section.

Let $\Upsilon_I=(|\Phi_I-\phi_I|>t\text{Lip}(f)/4)$ and $\Omega_I=(|r_I-\phi_I|>t\text{Lip}(f))$.    We claim that $\Phi_I$ is $2\lambda\text{Lip}(f)$-Lipschitz taking values in $[0, \lambda\text{Lip}(f)]$, so $\mathbb{P}(\Upsilon_I)\leqslant t/4\lambda$, $|\phi_I-r_I|\leqslant t/2$, and $\mathbb{P}\Big(\bigcup_{I\in T\setminus \Lambda_2}\Omega_I\Bigr)<1/2^l$. Since $\text{diam}(2^{lk})=1$ and $\text{Lip}(F\circ f)\leqslant \lambda\text{Lip}(f)$, we deduce that $\Phi_I$ takes values in $[0,\lambda\text{Lip}(f)]$.   Next let us show that $\Phi_I$ is $2\lambda\text{Lip}(f)$-Lipschitz.   Fix $\ee_1, \ee_2\in 2^{lk}$ and note that \begin{align*} \varrho_Y^f(\ee_1, I\ee_1) - \varrho_Y^f(\ee_2, I\ee_2) & \leqslant \varrho_Y^f(\ee_1, \ee_2)+\varrho_Y^f(\ee_2, I\ee_2)+\varrho_Y^f(I\ee_2, I\ee_1)-\varrho_Y^f(\ee_2, I\ee_2) \\ & = \varrho_Y^f(\ee_1, \ee_2)+ \varrho_Y^f(I\ee_1, I\ee_2)= 2\varrho_Y^f(\ee_1, \ee_2) \leqslant 2\lambda\text{Lip}(f) \partial(\ee_1, \ee_2).  \end{align*} By symmetry, we deduce that $\Phi_I$ is $2\lambda\text{Lip}(f)$-Lipschitz.    From this it follows that $$\mathbb{P}(\Upsilon_I) = \mathbb{P}\Bigl(|\Phi_I-\phi_I|> \frac{t}{8\lambda}(2\lambda \text{Lip}(f))\Bigr) \leqslant \alpha \exp(-\beta t k/8\lambda)<t/4\lambda.$$   Therefore \begin{align*} |\phi_I-r_I| & \leqslant \mathbb{E}1_{\Upsilon_I}|\phi_I-\Phi_I| + \mathbb{E}1_{\Upsilon_I^c}|\phi_I-\Phi_I|  \leqslant \lambda\text{Lip}(f)\mathbb{P}(\Upsilon_I)+ t\text{Lip}(f)/4 \leqslant t\text{Lip}(f)/2. \end{align*}  Therefore $\Omega_I\subset \Upsilon_I$ and $$\mathbb{P}\Bigl(\bigcup_{I\in T\setminus \Lambda_2} \Omega_I\Bigr) \leqslant \sum_{I\in T\setminus \Lambda_2} \mathbb{P}(\Upsilon_I) \leqslant (l+1)\alpha \exp(-\beta tk/8\lambda)<1/2^l.$$  From this and Remark \ref{density}, we may define $g:2^{lk}\times 2^l\to 2^{lk}$ by $g(\ee, \delta)(i)=\delta(j)\ee(i)$ when $i\in I_j$ and choose $\ee_0\in 2^{lk}$ such that $\{g(\ee_0, \delta): \delta\in 2^l\}\subset \bigcap_{I\in T\setminus \Lambda_2} \Omega_I^c$.

  Now define $h:2^l\to X$ by $h(\delta)=f(g(\ee_0, \delta))$.      Note that  \begin{align*} (1-\mu) l\max_{i\in [l], \delta\in 2^l} \varrho_Y^h(\delta, d_i\delta) & \leqslant \frac{1-\eta}{1+\eta}\cdot l\max_{I\in \Lambda_1, \delta\in 2^l} \Phi_I(g(\ee_0, \delta)) \leqslant lt\text{Lip}(f)+ \frac{1-\eta}{1+\eta}\cdot l r_I \\ &  \leqslant lt\text{Lip}(f)+(1-\eta)\Theta \text{Lip}(f)  \leqslant lt\text{Lip}(f)+r_{[lk]} \\ & \leqslant (l+1)t\text{Lip}(f)+\min_{\delta\in 2^l}\Phi_{[lk]}(g(\ee_0,\delta)) \\ & < (1+\mu)\min_{\delta\in 2^l} \Phi_{[lk]}(g(\ee_0,\delta)) = (1+\mu)\min_{\delta\in 2^l}\delta_Y^h(\delta, -\delta).   \end{align*} Here we are using the fact that since $$M \text{Lip}(f):=\min_{\delta\in 2^l} \Phi_{[lk]}(g(\ee_0, \delta)) \geqslant r_{[lk]} -t\text{Lip}(f) \geqslant ((1-\mu)\Theta - t)\text{Lip}(f) \geqslant (1-2\mu)\Theta \text{Lip}(f),$$ it follows from our choice of $t$ that $$(l+1)t\text{Lip}(f) +M\text{Lip}(f) < (1+\mu)M \text{Lip}(f).$$    Now our choice of $a$ and $\mu$ combined with Proposition \ref{sharp} yield that $\text{dist}(F\circ h)\leqslant D$.

We next show that for any $\delta_1, \delta_2\in 2^l$, $\vartheta \varrho_X^h(\delta_1, \delta_2) \leqslant \varrho_Y^f(\delta_1, \delta_2)$.   First observe that the map $\delta\mapsto g(\ee_0, \delta)$ is length preserving, so $\text{Lip}(h)\leqslant \text{Lip}(f)$. Therefore $$\varrho_X^h(\delta_1, \delta_2) \leqslant \text{Lip}(f) \partial(\delta_1, \delta_2).$$   Now let us observe that \begin{align*} \text{Lip}(F\circ h) & = \max_{i\in [l], \delta\in 2^l} l\varrho_Y^h(\delta, d_i\delta)\geqslant \min_{I\in \Lambda_1, \delta\in 2^l} l\Phi_I(g(\ee_0, \delta)) \\ & \geqslant \min_{I\in \Lambda_1} lr_I - lt\text{Lip}(f) \geqslant (1-\mu)\Theta \text{Lip}(f) -lt\text{Lip}(f)\geqslant (1-2\mu)\Theta \text{Lip}(f). \end{align*}    Since $h$ has distortion at most $D$, for any $\delta_1, \delta_2\in 2^l$, $$\varrho_Y^h(\delta_1, \delta_2) \geqslant  \frac{(1-2\mu)}{D}\cdot \Theta \text{Lip}(f) \partial(\delta_1, \delta_2) \geqslant  \vartheta \text{Lip}(f) \partial(\delta_1, \delta_2)\geqslant \vartheta\varrho_X^h(\delta_1, \delta_2).$$   An appeal to Lemma \ref{taco bell} finishes the proof.

\end{proof}


\begin{thebibliography}{HD}

\normalsize
\baselineskip=17pt

\bibitem{AM} D. Amir and V. D. Milman, \emph{Unconditional and symmetric sets in $n$-dimensional normed
spaces}, Israel J. Math. \textbf{37} (1980), 3-20.

\bibitem{Beauzamy} B. Beauzamy, \emph{Op\'{e}rateurs de type Rademacher entre espaces de Banach}, S\'{e}m Maurey-Schwartz, Expos\'{e}s 6 et 7, \'{E}cole Polyt\'{e}chnique, Paris (1975/76). 

\bibitem{BMW} J. Bourgain, V. Milman, H. Wolfson, \emph{On type of metric spaces},  Trans. Amer. Math. Soc., \textbf{294} (1986), no. 1, 295-317. 

\bibitem{CDK} R.M. Causey, S. Draga, T. Kochanek, \emph{Operator ideals and three-space properties of asymptotic ideal seminorms}, submitted. 

\bibitem{ChavezDominguez} J.A. Ch\'{a}vez-Dominguez, \emph{Lipschitz $(q,p)$-mixing operators},  Proc. Amer. Math. Soc., \textbf{140} (2012), no. 9, 3101-3115.

\bibitem{Enflo} P. Enflo, \emph{On infinite dimensional topological groups}, S\'{e}minaire sur la G\'{e}om\'{e}trie des Espaces de Banach (1977–1978), Exp. No. 10–11, 11 pp., École Polytech., Palaiseau, 1978. 

\bibitem{EFW} A. Eskin, D. Fisher, K. Whyte, \emph{Coarse differentiation of quasi-isometries I: Spaces not quasi-isometric to Cayley graphs}, Ann. of Math., \textbf{176} (2012), no. 1, 221-260. 

\bibitem{FJ} J. Farmer, W.B. Johnson, \emph{Lipschitz $p$-summing operators}, Proc. Amer. Math. Soc., \textbf{137} (2009), no. 9, 2989-2995. 

\bibitem{Harper} L. Harper, \emph{Optimal numberings and isoperimetric problems on graphs}, J. Combinatorial
Theory \textbf{1} (1966), 385-393.

\bibitem{Hinrichs} A. Hinrichs, \emph{Operators of Rademacher and Gaussian subcotype}, J. Lond. Math. Soc., \textbf{63} (2001), no. 2, 453-468.

\bibitem{JS} W.B. Johnson, G. Schechtman, \emph{Diamond graphs and super-reflexivity}, J. Topol.
Anal., \textbf{1} (2009), no. 2, 177-189

\bibitem{KOS} H. Knaust, Th. Schlumprecht, E. Odell,  \emph{On Asymptotic structure, the Szlenk
index and UKK properties in Banach spaces}, Positivity \textbf{3} (1999), 173-199.

\bibitem{Naor} A. Naor, \emph{An introduction to the Ribe program}, Jpn. J. Math, \textbf{7}, (2012), no. 2, 167-233. 

\bibitem{Pisier} G. Pisier, \emph{Martingales with values in uniformly convex Banach spaces}, Israel J. Math., \textbf{
20}(1975), no. 3, 326-350.

\bibitem{Ribe} M. Ribe, \emph{On uniformly homeomorphic normed spaces}. Ark. Mat., \textbf{14} (1976), no. 2, 237-
244.


\bibitem{Wenzel} J. Wenzel, \emph{Uniformly convex operators and martingale type}, Rev. Mat. Iberoam., \textbf{18}  (2002), no. 1, 211-230. 


\end{thebibliography}
\end{document}